\documentclass[11pt]{amsart} 
\usepackage{amsmath, amssymb, latexsym, tikz}
\usepackage{caption,ifthen,url,fancyhdr,cite}
\usepackage[foot]{amsaddr}    
\usetikzlibrary{arrows,shapes,positioning}
\tikzstyle{gvert}=[draw,circle,line width=0.5pt,minimum size=10pt,inner sep=0pt]
\tikzstyle{ivert}=[circle,line width=0.5pt,minimum size=10pt,inner sep=0pt]
\tikzstyle{edge} = [draw,>=latex,->]
\usepackage{libertine} 
\usepackage{longtable} 

\usepackage[a4paper,textwidth=16.4cm,textheight=22.5cm,centering]{geometry}

\newtheoremstyle{mythm}                   %change thm environm. 
{6pt}%space above 
{6pt}%space below
{\it}%body font
{}%indent amount
{\bf}%theorem head font
{.}%punctation after theorem head
{.5em}%space after thm head
{}%thm head spec

\newtheoremstyle{mydef}                   %change def environm. 
{6pt}%space above
{6pt}%space below
{}%body font
{}%indent amount
{\bf}%theorem head font
{.}%punctation after theorem head  
{.5em}%space after thm head
{}%thm head spec

\newtheoremstyle{myrem}                   %change rem environm. 
{6pt}%space above
{6pt}%space below
{}%body font
{}%indent amount
{\bf}%theorem head font
{.}%punctation after theorem head
{.5em}%space after thm head
{}%thm head spec

\theoremstyle{mythm}      
\newtheorem{theorem}{Theorem}[section]
\newtheorem{proposition}[theorem]{Proposition}
\newtheorem{lemma}[theorem]{Lemma}
\newtheorem{corollary}[theorem]{Corollary}        
\newtheorem{conjecture}[theorem]{Conjecture} 

\theoremstyle{mydef}      

\newtheorem{example}[theorem]{Example}

\theoremstyle{myrem}
\newtheorem{remark}[theorem]{Remark}
\numberwithin{equation}{section}

\newcounter{ithmcount}

\newenvironment{iprf}{\begin{list}{{\rm
	\alph{ithmcount})}}{\usecounter{ithmcount}\labelwidth-5pt
      \leftmargin0pt \topsep3pt \itemsep1pt \parsep2pt}}{\end{list}}

\newenvironment{ithm}{\begin{list}{{\rm \alph{ithmcount})}}{\usecounter{ithmcount}\labelwidth18pt
      \leftmargin18pt \topsep3pt \itemsep1pt \parsep2pt}}{\end{list}}

\renewcommand{\leq}{\leqslant} 
\renewcommand{\geq}{\geqslant}
\newcommand{\ra}{\rightarrow}

\newcommand{\ms}{\mapsto}

\renewcommand{\split}{\rtimes}
\newcommand{\im}{{\rm{im}}}
\newcommand{\Aut}{{\rm Aut}}

\newcommand{\gal}{{\rm gal}}
\newcommand{\N}{\mathbb N}

\newcommand{\Z}{\mathbb Z}
\newcommand{\Q}{\mathbb Q}
\newcommand{\K}{\mathbb K}
\renewcommand{\L}{\mathbb L}
\renewcommand{\O}{\mathcal O}
\newcommand{\p}{\mathfrak p}

\newcommand{\newG}{\mathcal{G}}
\newcommand{\dep}{{\rm dep}}
\newcommand{\T}{\mathcal T}
\newcommand{\B}{\mathcal B}
\newcommand{\DD}{\mathcal D}
\newcommand{\R}{\mathcal R}
\renewcommand{\S}{\mathcal S}

\newcommand{\Hom}{{\rm{Hom}}}
\renewcommand{\SS}{\mathbb S}
\newcommand{\TT}{\mathbb T}

\newcommand{\II}{\mathcal I}

\newcommand{\G}{\mathcal G}
\newcommand{\U}{\mathcal U}

\newcommand{\hatn}{{\hat{n}}}

\newcommand{\defd}{d}
\newcommand{\defc}{c}
\newcommand{\defell}{\ell}

\begin{document}

\title{Galois trees in the graph of $\pmb{p}$-groups of maximal class}
\subjclass[2000]{}
\author[A. Cant]{Alexander Cant}
\address[Cant, Eick, Moede]{Institut f\"ur Analysis und Algebra, Technische Universit\"at Braunschweig, Germany}
\author[H. Dietrich]{Heiko Dietrich}
\address[Dietrich]{School of Mathematics, Monash University, Clayton VIC 3800, Australia}
\author[B. Eick]{Bettina Eick}
\author[T. Moede]{Tobias Moede}
\email{\rm a.cant@tu-bs.de, heiko.dietrich@monash.edu,  beick@tu-bs.de, t.moede@tu-bs.de}
\keywords{$p$-groups, coclass theory, coclass graphs, maximal class}
\date{\today}

\begin{abstract}
The investigation of the graph $\newG_p$ associated with the finite $p$-groups of
maximal class was initiated by Blackburn (1958) and became a deep and
interesting research topic since then. Leedham-Green \& McKay (1976--1984)
introduced skeletons of $\newG_p$, described their importance for the structural
investigation of $\newG_p$ and exhibited their relation to algebraic number theory.
Here we go one step further: we partition the skeletons into so-called
Galois trees and study their general shape. 
In the special case $p \geq 7$ and $p \equiv 5 \bmod 6$, we show that they 
have a significant impact on the periodic patterns of $\newG_p$ conjectured 
by Eick, Leedham-Green, Newman \& O'Brien (2013). In particular, we use 
Galois trees to prove a conjecture by Dietrich (2010) on these periodic patterns.
\end{abstract}

\maketitle  

\section{Introduction}
\label{intro}

\noindent 
The classification of $p$-groups of maximal class is a long-standing research
problem. It was first investigated by Blackburn \cite{Bla58} and later continued
in various publications. In a broader context, this research problem is part of 
coclass theory. The monograph by Leedham-Green \& McKay \cite{LGM02} 
provides an introduction to this research area, whereas their work \cite{LMc76, LMc78a, LMc78, LMc84}
puts particular emphasis on the maximal class case. We first introduce some notation 
in order to describe the theory and our main results.
\smallskip

\noindent {\bf The graph $\newG_p$.}
The $p$-groups of maximal class can be visualised by a graph $\newG_p$: its 
vertices correspond one-to-one to the isomorphism type representatives of 
finite $p$-groups of maximal class, and there is an edge $G \ra H$ if and 
only if $H/\gamma(H) \cong G$, where $\gamma(H)$ denotes the last 
non-trivial term in the lower central series of $H$. The graph $\newG_p$
is an infinite graph and the broad structure of $\newG_p$ is as follows: It consists of an 
isolated vertex corresponding to the cyclic group of order $p^2$, and
an infinite tree $\T_p$ with root corresponding to the elementary abelian 
group of order $p^2$. The tree $\T_p$ has a unique infinite path $S_p(2)
\ra S_p(3) \ra \cdots$ starting at its root. Each group $S_p(n)$ has
order $p^n$ and is isomorphic to a lower central series quotient
of the infinite pro-$p$-group of maximal~class. The graph $\newG_p$ was fully 
determined for $p \leq 3$ by Blackburn \cite{Bla58}, and it is experimentally 
well investigated for $p=5$, see \cite{DEF08,New90,Die10}. For $p \geq 7$ its 
detailed structure is still widely unknown and our results will be focused on that case.  
\smallskip

\noindent {\bf Branches.}
A group $H$ is a {\em descendant} of $G$
in $\T_p$ if there is a path from $G$ to $H$ in $\T_p$. If $H$ has distance $m$ 
to $G$, then $H$ is an {\em $m$-step descendant} of $G$; if $m=1$ then $H$ is an 
{\em immediate descendant} of $G$; we define {\em $m$-step ancestors} 
analogously. A group is {\em capable} if it has immediate descendants.
The \emph{branch} $\B_p(n)$ of $\T_p$ is the full subgraph of $\T_p$ 
consisting of all descendants of $S_p(n)$ that are not descendants of 
$S_p(n+1)$. Thus, $\T_p$ consists of the  branches $\B_p(2), \B_p(3), 
\ldots$ which are connected by its infinite path. 

The {\em depth} of a group $G$ in $\B_p(n)$ is its distance to the root 
$S_p(n)$. The {\em depth} $\dep(\B_p(n))$ is the maximum of the depths 
of groups in $\B_p(n)$. Blackburn's classification 
\cite{Bla58} implies that the depth of $\B_p(n)$ is bounded for $p \leq 3$.  
For $p \geq 5$, the depth of $\B_p(n)$ grows linearly in $n$. More 
precisely, for $p \geq 5$ Dietrich \cite{Die09} proved 
that with $c=2p-8$ for $p \geq 7$ and $c=4$ for $p=5$ the depth satisfies $n-\defc \leq \dep(\B_p(n)) \leq n+\defc-3$.

\pagebreak

\noindent {\bf Periodic patterns.}
Throughout we write $d=p-1$. A {\em pruned branch} $\B_p(n,k)$ is the subtree of $\B_p(n)$ containing all 
groups of depth at most $k$ in $\B_p(n)$. Extending the results of du Sautoy \cite{DuS01} 
and Eick \& Leedham-Green \cite{ELG08}, Dietrich \cite{Die09} proved that
there exists $n_0=n_0(p)$ such that for all $n \geq n_0$ and $c$ as defined above there exists a graph isomorphism 
\[ \iota \colon \B_p(n+\defd, n-\defc) \ra \B_p(n, n-\defc).\]
This is called the {\em first periodicity} of the graphs of $p$-groups of
maximal class. It remains to study $\B_p(n+\defd)$ at `large' depth at least $n-\defc$.
For $G$ in $\B_p(n)$ we denote by $\DD(G)$ the tree of all descendants of $G$ in
$\B_p(n)$ and we write $\DD(G,m)$ for its pruned version. Two groups $G$ and 
$H$ in $\T_p$ are {\em twins} if $\DD(G) \cong \DD(H)$. They are {\em $m$-step
twins} if $\DD(G,m) \cong \DD(H,m)$. The following conjecture is a main
open problem in the theory of $p$-groups at current and a driving force of
current research in this area, see \cite{Die10, ELNO13}.

\begin{conjecture} 
\label{conj1}
Let $p \geq 5$. There exists $n_0 = n_0(p)$ such that for all $n \geq n_0$ and each group 
$G$ at depth $n-\defc$ in $\B_p(n+\defd)$ there exists a twin $H$ at depth $n-\defc-\defd$ in 
$\B_p(n)$.
\end{conjecture}

If this conjecture is true, then $\G_p$ is determined  by a finite subgraph and two periodicities. The available computational evidence for small primes supports Conjecture \ref{conj1}, 
but unfortunately this evidence is rather thin. If one wants to prove
Conjecture \ref{conj1}, then a construction for twins is needed. A first 
idea was to use $H = \iota(K)$ as potential twin, where $K$ is a $\defd$-step ancestor of $G$. 
This the reason why twins have been called {\em periodic parents} in earlier work. 
This idea turned out to be wrong, see Section \ref{secperiod} for details. 

\noindent {\bf Skeletons.}
The {\em skeleton} $\S_p(n)$ 
is the full subtree of $\B_p(n)$ consisting of all capable groups of depth at most $n-\defc$ in $\B_p(n)$. 
Leedham-Green \& McKay introduced and studied these groups in a more general setting under the name  {\em constructible
groups}: they described a construction of these groups based on algebraic number theory and they also translated the 
isomorphism problem of these groups to a problem in number theory; we recall some of this theory in Section \ref{antskelgroups}.

It is the central aim of this paper to exhibit new structural results for the graph $\newG_p$ and its 
skeletons, with a view towards a construction of twins. From now on, $p\geq 7$ is a prime and write 
\[\defd=p-1,\quad \defc=2p-8,\quad\text{and}\quad \defell=(p-3)/2.\]

\section{Main results}
\noindent We define the {\em Galois order} $\gal(G)$ of a $p$-group $G$ as the $p'$-factor of its 
automorphism group order. If $n \geq 3$, then $\gal(S_p(n))=\defd^2$, and $\gal(G)$ 
divides $\defd$ for every other group $G$ in $\S_p(n)$. Further, if $H$ is a 
descendant of $G$, then $\gal(H)$ divides $\gal(G)$; see Theorem \ref{isomautpr}d). We call a subtree $\R$ of$\S_p(n)$ a {\em Galois tree} with {Galois order} $h$ if $\gal(G)=h$ for
all $G$ in $\R$ and $\R$ is maximal with this property, that is, neither an ancestor of the root of $\R$ nor a descendant of a leaf of $\R$ has Galois order $h$. By
definition the Galois trees (with all possible Galois orders) form a partition of $\S_p(n)$. 

The work by Dietrich \cite[Section 6]{Die10} suggests that the investigation of Galois trees may provide a way to determine twins. Motivated by this, the Galois trees with Galois 
order $\defd$ have been fully determined by Dietrich \& Eick \cite{DEi17}: each skeleton $\S_p(n)$ consists of $\defell$ such Galois trees, with all roots at depth 
$1$ and all leaves at depth $\dep(\S_p(n))$. Furthermore, all of these trees are {\em distance regular}, that is, the number of immediate descendants of a group in such a Galois 
tree depends on the distance to its root only. The work of this paper is guided by an investigation of Galois trees for Galois orders $h\mid d$.

We start by revisiting and refining the construction of skeleton groups by Leedham-Green \& McKay. Our new construction simplifies the isomorphism problem and also allows us 
to read off information about the Galois order of the groups. Stating these new results in detail requires more notation, which is why we postpone this to  Theorem \ref{thmGFP} 
and Proposition \ref{propIsoGT}. Both of these results are key ingredients in the proofs of our three main results stated below. The first of these, Theorem \ref{skelbottom}, shows 
that the Galois trees for arbitrary Galois order also extend to the full depth of the skeleton.

\begin{theorem} \label{skelbottom}
Let $p \geq 7$ and $n \geqslant \max \{ \defc, 8 \}$. Every leaf of a Galois tree in $\S_p(n)$  has depth $\dep(\S_p(n))$.
\end{theorem}

A Galois tree $\mathcal{R}$  has \emph{ramification level} $e$ if there is a group in $\mathcal{R}$ at depth $e$ in $\S_p(n)$ that has more than one immediate descendant in $\mathcal{R}$. 
Our second main result Theorem \ref{thmGalprops} shows that for $p\equiv 5 \bmod 6$ the ramification levels and depths of roots of Galois trees occur periodically with periodicity $d$. We prove Theorem
\ref{thmGalprops} in Sections \ref{isorev} and \ref{galroots}. The distinction between $p \equiv 1 \bmod 6$ and $p \equiv 5 \bmod 6$ dates back to work of Leedham-Green \& McKay and, in particular,
the assumption $p \equiv 5 \bmod 6$ implies a key property of certain stabilisers, see \cite[Theorem 5.1]{Die10} and Section \ref{secperiod}. 

\begin{theorem} \label{thmGalprops}
Let $p \geq 7$ with $p \equiv 5 \bmod 6$ and $n \geqslant \max \{ \defc, 8 \}$. Let $\R$ be a Galois tree in $\S_p(n)$ with Galois order $h$. There exists $e_0=e_0(p)$ such that for 
all $e_0 \leq e < \dep(\S_p(n))$ the following hold:
\begin{ithm}  
\item If $G$ is a group at depth $e$ in $\R$, then the number of immediate descendants of $G$ in $\R$ is a power of $p$.
\item Let $e+\defd<\dep(\S_p(n))$. Then $\R$ has a ramification level $e$ if and only if it has a ramification level $e+\defd$. 
 \item Let $e+\defd<\dep(\S_p(n))$. If the root of $\R$ has depth $e$ in $\S_p(n)$, then there is a Galois tree with Galois order $h$ and root at depth $e+\defd$ in $\S_p(n)$.
\end{ithm}
\end{theorem}

Our new description of skeleton groups can also be used to prove a first existence result for twins. As before, we assume $p\equiv 5 \bmod 6$. We prove Theorem \ref{thm_newperiod_short}  in Section \ref{secperiod}.

\begin{theorem} \label{thm_newperiod_short}
Let $p \geq 7$ with $p \equiv 5 \bmod 6$. There exists $n_0 = n_0(p)$ so that for all $n \geq n_0$ every skeleton group $G$ at depth $\dep(\S_p(n))$ in $\B_p(n+\defd)$ has a $\defd$-step
twin  at depth $\dep(\S_p(n))-d$ in $\B_p(n)$.
\end{theorem}

This result provides a first serious step towards a construction of twins. It is also the first periodicity result that yields a full description of the
growth of the (pruned) branches of a coclass tree with growing depths and widths, see \cite[Appendix A.2]{DEi17} for an overview of known periodicity results.

We have used computational methods to explore further properties of Galois trees. The examples exhibited in Appendix \ref{examples} allow us to make the following observations.

\begin{remark}
\begin{iprf}
 \item[\rm a)] 
 Galois trees with Galois order $d$ are distance regular; see \cite{DEi17}. This result does not extend to Galois trees with arbitrary Galois order; see Figure \ref{fignondistreg}.
 \item[\rm b)] If $p \equiv 5 \bmod 6$, then the number of immediate descendants in a Galois tree is a power of $p$. It seems that this does not extend to $p \equiv 1 \bmod 6$; see Figure \ref{fignondistreg}.
 \item[\rm c)] All Galois trees with Galois order $\defd$ have roots of depth $1$. For arbitrary Galois order, there seems to be no bound to the depths of the roots of Galois trees; see Figure \ref{figs724}.
\end{iprf}
\end{remark}

%%%%%%%%%%%%%%%%%%%%%%%%%%%%%%%%%%%%%%%%%%%%%%%%%%%%%%%%%%%%%%%%%%%%%%%%%%%%%
\section{Preliminaries}\label{antskelgroups}

\noindent We recall some results from algebraic number theory and then the description of skeleton groups and their isomorphism problem.
Along the way, we introduce the notation that is fixed throughout this paper. To assist the reader with the notation, Appendix \ref{glossary} provides a short glossary.

\enlargethispage{6ex}

\noindent {\bf Algebraic number theory.}
Most of the following results can be found in standard books such as Neukirch \cite[Chapter~II]{Neu11}, but we also refer to \cite{DEi17,DS19} for a 
recent treatment and more details. Let $\Q_p$ denote the $p$-adic rational numbers with $p$-adic integers $\Z_p$, 
and let $\theta$ be a primitive $p$-th root of unity over $\Q_p$. The field
$\K = \Q_p(\theta)$ is an extension of degree $\defd$ over
$\Q_p$ with $\Q_p$-basis $\{ 1, \theta, \ldots, \theta^{\defd-1} \}$. The maximal 
order of $\K$ is $\O = \Z_p[\theta]$. It contains a unique maximal ideal $\p = (\theta-1)$, 
which yields the unique series of ideals given by the ideals $\p^m$, where $\p^0=\O.$ 
One can extend this definition  to fractional ideals $\p^m$ 
for all $m \in \Z$. An important property is that $p\p^m=\p^{m+\defd}$ for all $m$. 
Below we write $(\p^m)^\defell$ for the direct product of
$\defell$ copies of $\p^m$, and not for the ideal $\p^{m\defell}$.
The group of units $\U$ of $\O$ has the form $\U = \langle \omega \rangle
\times \langle \theta \rangle \times \U_2$, where $\omega$ is a primitive 
$\defd$-th root of unity in $\Z_p$ and $\U_2 = 1 + \p^2$. 

For $i\in\Z$ coprime to $p$ we define the automorphism $\sigma_i \colon \K \ra \K,\; \theta \ms \theta^i$,
so that $\G = \{ \sigma_i : 1 \leq i \leq \defd \}$ is the Galois group of $\K$; 
it is cyclic of order $\defd$. Throughout, we fix a generator $\sigma$ of $\G$ and set %\enlargethispage{3ex}
$\tau=\sigma^k$ for a fixed divisor $k$ of $\defd$; note that $\tau$ has order $h=\defd/k$ in $\G$ and $d=hk$ holds.

The automorphism $\tau$ induces a diagonalisable $\Z_p$-linear map on $\p^2$; there are $h$ 
distinct eigenspaces of $\Z_p$-dimension $k$ each, with eigenvalues $\{ \omega^{ki} : 0 \leq i \leq \defd-1 \}$. The map $\chi \colon \U \ra \U,\; u \ms u \tau(u)^{-1}$ is a group homomorphism with
$\U = \ker(\chi) \times \im(\chi)$. If $h < \defd$, then $\ker(\chi) = \langle \omega \rangle \times (1 + \mathcal{F})$ where $\mathcal{F} = \{ a \in \p^2 : \tau(a) = a\}$ is a $\Z_p$-subspace 
of dimension $k$ in $\O$; if $h=\defd$, then  $\ker(\chi)=\Z_p^\ast$. The results above are proved analogously to \cite[Lemmas 2.1 \& 2.2]{DEi17}, which cover the case $h=\defd$.

\smallskip

\noindent {\bf Skeleton groups.}
Skeleton groups play a crucial role in almost all recent treatments of coclass graphs. 
These groups were first described by Leedham-Green \& McKay using the name constructible groups and we refer to \cite{LGM02} for details. 
For a recent discussion of skeleton groups we refer to \cite{DS19}. Below we follow a slight modification that has already been used in \cite{DEi17}; proofs can be found in said references.

Let $\Hom=\Hom_{\theta}(\O \wedge \O, \O)$ be the group of homomorphisms $\O \wedge \O \ra \O$ that are compatible with multiplication by $\theta$.
The image of $f \in \Hom$ is an ideal in $\O$, so $\im(f)=\p^{n-1}$ for some $n \in \N$. Given $e \in \{0, \ldots, \dep(\S_p(n))\}$ and $f \in \Hom$ with
image $\p^{n-1}$, one defines a multiplicative group structure on $\O / \p^{n+e-1}$ via
$(a + \p^{n+e-1}) \circ (b + \p^{n+e-1}) = (a+b + \tfrac{1}{2} f(a \wedge b) + \p^{n+e-1})$.

The resulting group $M_{n,e}(f)$ is a group of order $p^{n+e-1}$ and class $2$. The group $\langle \theta \rangle$ 
is cyclic of order $p$;  since $f$ is compatible with multiplication by $\theta$, it follows that $\langle \theta \rangle$  acts 
naturally on $M_{n,e}(f)$. For $n \geq \max\{\defc, 8\}$ and $0 \leq e \leq \dep(\S_p(n))$ a homomorphism $f\in \Hom$ 
with image $\p^{n-1}$ defines a group of depth $e$ in $\S_p(n)$ via $M_{n,e}(f) \rtimes \langle \theta \rangle$ and, conversely, every group 
of depth $e$ in $\S_p(n)$ can be obtained as $M_{n,e}(f) \rtimes \langle \theta \rangle$ for some $f \in \Hom$ with image $\p^{n-1}$;
see e.g.\ \cite[Corollary 3.5]{DS19}.

The following is taken from \cite[Sections 8.2 \&  8.3]{LGM02}, see also \cite[Section 4.1]{DEi17}. The group $\Hom$ is a free $\O$-module and has a natural 
basis $\{\TT_1,\ldots,\TT_\ell\}$ such that the homomorphism $x_1 \TT_1 + \ldots + x_\ell \TT_{\ell}$ has image $\p^n$ if and only if each $x_i\in\p^n$ and 
at least one $x_j\notin \p^{n+1}$.  Leedham-Green \& McKay noted that this basis is not well adapted to the isomorphism problem, so they introduced a different 
generating set $\{\SS_2,\ldots,\SS_{\ell+1}\}$. (The precise definitions can be found on \cite[p.\ 235]{DEi17}; we do not repeat them here because we do not use them explicitly and want to avoid unnecessary notation.) For every $f\in \Hom$ there is a unique tuple $x=(x_1,\ldots,x_\ell)\in \K^\ell$ such that 
$f=\SS(x)=x_1 \SS_2 + \ldots + x_\ell \SS_{\ell+1}$ and we define the corresponding skeleton group as
\[ C_{n,e}(x)=M_{n,e}(\SS(x)) \rtimes \langle \theta \rangle. \]
However, $x_i\notin \O$ is possible, and the base change matrix $B$ from $\{\SS_2,\ldots,\SS_{\ell+1}\}$ to $\{\TT_1,\ldots,\TT_\ell\}$ is not invertible over 
$\O$. We define  \[\Gamma_n = (\p^{n-1})^\ell B^{-1}\quad\text{and}\quad \Delta_n = \Gamma_n \setminus \Gamma_{n+1},\] so that tuples in $\Delta_n$ parametrise the homomorphisms 
that have image equal to $\p^{n-1}$. To discuss the isomorphism problem for the groups $C_{n,e}(x)$, the following is required. For $i \in \{ 2, \ldots, \defell+1\}$ let $\rho_i \colon \U \ra \U$\linebreak be defined by $\rho_i(u)= u^{-1} \sigma_i(u) \sigma_{1-i}(u)$. It is shown in \cite[Lemma~4.3]{DEi17} that  $(u,\varphi)\in \U \split \G$ acts 
on $x=(x_1,\ldots,x_\defell)\in \K^\defell$ via
\begin{eqnarray}\label{eqAct} (u, \varphi ) (x) =  
      (\rho_2(u)^{-1} \varphi(x_1), \ldots, \rho_{\defell+1}(u)^{-1} \varphi(x_\defell)).
\end{eqnarray}

The next theorem summarises some  key results for skeleton groups and their isomorphism problem.\enlargethispage{7ex}

\begin{theorem} 
\label{isomautpr}
Let $n\geq \max\{\defc,8\}$ and $0\leq e\leq \dep(\S_p(n))$. Let $x,y\in\Delta_n$.
\begin{ithm}
 \item  The map $x \ms C_{n,e}(x)$ yields an onto map from $\Delta_n$ to the groups of depth $e$ in $\S_p(n)$.
 \item We have $C_{n,e}(x) \cong C_{n,e}(y)$ if and only if $(u, \varphi)(x) \equiv y \bmod \Gamma_{n+e}$ for some $(u, \varphi) \in \U \split \G$.
 \item We have $h\mid \gal(C_{n,e}(x))$  if and only if $(u, \tau)(x) \equiv x \bmod \Gamma_{n+e}$ for some $u\in \U$.
 \item If $G$ has depth $e\geq1$ in $\S_p(n)$, then $\gal(G)$ divides $d$ and the Galois order of every descendant of $G$ in $\S_p(n)$ divides $\gal(G)$.
\end{ithm}
\end{theorem}

\begin{proof}
Part a) is proved in  \cite[Sections 8.2 \&  8.3]{LGM02} and \cite[Section 4.1]{DEi17}. 
A proof for part b) can be found in \cite[Section 4.2]{DEi17}, based on ideas of Leedham-Green \& McKay. 
By \cite[Lemma 5.4]{Die10}, the $p'$-part of $|\Aut(C_{n,e}(x))|$ is the same as 
the order of the image of the projection  ${\rm Stab}_{\mathcal{U}\rtimes\G}(x+\Gamma_{n+e})\to \G$; 
thus, $h$ divides the Galois order of $C_{n,e}(x)$ if and only if there is some $u\in\mathcal{U}$ with 
$(u,\tau)(x)\equiv x\bmod \Gamma_{n+e}$; see \cite[Theorem 4.4]{DEi17} for additional information on 
$\Aut(C_{n,e}(x))$. This implies c) and d).
\end{proof}

%%%%%%%%%%%%%%%%%%%%%%%%%%%%%%%%%%%%%%%%%%%%%%%%%%%%%%%%%%%%%%%%%%%%%%%%%%%%%
\section{A local to global principle}
\label{localglobal}

\noindent Theorem~\ref{isomautpr} describes the groups of depth $e$ in $\S_p(n)$ whose Galois order is divisible by $h$ via fixed
points under some action modulo $\Gamma_{n+e}$; that is, they are
fixed points in a \emph{local} setting. The  aim of this section is to prove
the following theorem, which shows that they can also be described via {\em global} fixed points. 
This result is our first main result and a central ingredient for the proofs of our main theorems.

\begin{theorem} \label{thmGFP}
Let $p \geqslant 7$ and $n \geqslant \max \{ \defc, 8 \}$.
Every group in $\S_p(n)$ is defined by a `global Galois fixed point': if $G$ has depth $e$ in $\S_p(n)$ 
and Galois order $h$, then $G\cong C_{n,e}(x)$ for some $x \in \Delta_n$ that is fixed
by some $(v,\varphi)\in \U\rtimes \G$, where $\varphi$ has order $h$ and $v \in \Z_p$ is a $\defd$-th root of unity.
\end{theorem}

Throughout this section, let $n \geq \max\{ \defc, 8\}$ and 
$0 \leq e \leq \dep(\S_p(n))$. Recall that $\tau = \sigma^k$ where $d = hk$, so  $\tau$ generates 
the subgroup of order $h$ in the Galois group~ $\G$. In some results below we use that the homomorphism $\chi \colon
\U \ra \U,\; u \ms u \tau(u)^{-1}$ induces the decomposition $\U = \ker(\chi) \times \im(\chi)$.

\begin{lemma}
\label{equiv1}
If $G$ has depth $e$ in  $\S_p(n)$, then  $h \mid \gal(G)$ if and only if $G \cong C_{n,e}(x)$  
for some  $x \in \Delta_n$ such that $(v, \tau)(x) \equiv x \bmod \Gamma_{n+e}$ for some $v\in \ker(\chi)$.
\end{lemma}

\begin{proof}
Theorem~\ref{isomautpr}c) shows that $G \in \S_p(n)$ and $h \mid \gal(G)$ if and only if $G \cong C_{n,e}(y)$ for some $y\in\Delta_n$ 
such that  $(u, \tau)(y) \equiv y \bmod \Gamma_{n+e}$ for some $u\in\U$. Write $u = v w$ with $v \in \ker(\chi)$ and 
$w \in \im(\chi)$, that is, $w^{-1}=s \tau(s)^{-1}$ for some  $s \in \U$. Now define $x = (s,1)(y)$; Theorem~\ref{isomautpr}b)
shows that $G\cong C_{n,e}(x)$ and, by construction, we have 
\[x \equiv(s,1)(u,\tau)(y) \equiv (s,1)(u,\tau)(s,1)^{-1}(x) \equiv (s u \tau(s)^{-1},\tau)(x) \equiv (v,\tau)(x)\bmod \Gamma_{n+e}.\qedhere\]
\end{proof}

As a next step, we investigate the action of $\tau$ on $\Gamma_n$. Let $\L = \{ a \in \K : \tau(a) = a\}$ be the fixed field
in $\K$ under the action of $\tau$. %Galois theory implies that $[\K:\L]= h$ and $[\L:\Q_p] = k$. 
Let $\O_\L = \O \cap \L$ the maximal order in 
$\L$ and let $\U_\L = \U \cap \L$ be its unit group; note that $\U_\L = \ker(\chi)$. For $i \in \{0, \ldots, \defd-1\}$ let
\[ \Sigma_{n,i} = \{ x \in \Gamma_n : (1,\tau)(x) = \omega^i x \} \]
and
\[ \Omega_{n,i} = \{ x \in \Delta_n : (1,\tau)(x) = \omega^i x \}. \]
Note that $\Omega_{n,i} = \Sigma_{n,i} \setminus \Sigma_{n+1,i}$ for all $i \in \{0, \ldots, \defd-1\}$.

\begin{lemma}
\label{lattice}
Let $\II = \{ (kj) \bmod \defd : 0 \leq j \leq \defd-1\}$. For each $i \in \II$ the set
$\Sigma_{n,i}$ is an $\O_\L$-sublattice of $\Gamma_n$ such that $(u,1)(x) \in
\Sigma_{n,i}$ for each $x \in \Sigma_{n,i}$ and  $u \in \U_\L$. One can decompose
\[ \Gamma_n = \bigoplus\nolimits_{i \in \II} \Sigma_{n,i}.\]
\end{lemma}

\begin{proof}
For $h=\defd$ this is \cite[Lemma 5.5]{DEi17}. It follows readily from that lemma that $\Sigma_{n,i}$ is an $\Z_p$-sublattice of $\Gamma_n$; 
note that $(1,\sigma)(x) = \omega^i x$ implies that $(1,\tau)(x) = \omega^{ki} x$, and 
this also asserts that $\Gamma_n = \bigoplus_{i \in \II} \Sigma_{n,i}$.
It remains to show that each $\Sigma_{n,i}$ is invariant under the
scalar action of $\O_\L$ and the twisted action of $\U_\L$. Let $x \in \Sigma_{n,i}$. If $a\in\O_\L$, then $(1,\tau)(ax)
= \tau(a) (1,\tau)(x) = a \omega^i x = \omega^i (ax)$, so $ax\in \Sigma_{n,i}$. %If $u\in \U_L$, then 
For $u \in \U_\L$ we have $(u,1)(x) = 
(\rho_2(u)^{-1} x_1, \ldots, \rho_{\defell+1}(u)^{-1} x_\ell)$. Since $u\in \U_\L$, each $\rho_j(u) \in \U_\L$, so $(1,\tau)((u,1)(x)) = (u,1) ((1,\tau)(x)) 
= (u,1) (\omega^i x) = \omega^i ((u,1)(x))$, as claimed. 
\end{proof}

Next, we give a proof for a number-theoretic condition for the existence of global fixed points.

\begin{lemma} \label{xiny}
We have\, $\Omega_{n,i} \neq \emptyset$ if and only if $(1,\tau)$ has an eigenvector in $\Delta_1$ with eigenvalue $\omega^{i+k(-n+1)}$; in particular, 
the eigenvalues of $(1,\tau)$ on $\Delta_1$ are $\{ \omega^{k(-2j+1)} : j = 1,\ldots,\defell \}$, so $\Omega_{n,i} \neq \emptyset$ if and only if 
$i \equiv k(n-2j) \bmod \defd$ for some $j=1,\ldots,\defell$.
\end{lemma}

\begin{proof}
For $h=\defd$ this is \cite[Lemma 5.7]{DEi17}. Let $x \in \Delta_1$ be an eigenvector of $(1,\tau)$ with eigenvalue $\omega^{i+k(-n+1)}$. By \cite[Lemma 2.1]{DEi17}
there exists an eigenvector $a \in \p^{n-1} \setminus \p^n$ of $\sigma$ with eigenvalue $\omega^{n-1}$. Since $\tau=\sigma^k$,  it follows that $a$ is an eigenvector
of $\tau$ with eigenvalue $\omega^{k(n-1)}$. We have $ax \in \Delta_n$ and
\[ (1,\tau)(ax)  = \omega^{k(n-1)} a \omega^{i+k(-n+1)} x
                      = \omega^i (ax),\]
so $ax \in \Omega_{n,i}$. Conversely, if $x \in \Omega_{n,y}$ is an eigenvector of $(1,\tau)$ with eigenvalue $\omega^i$, then choose an eigenvector 
$a \in \p^{-n+1} \setminus \p^{-n+2}$ of $\tau$ with eigenvalue $\omega^{k(-n+1)}$, and observe that $ax \in \Delta_1$ satisfies
\[ (1,\tau)(ax) 
              = \omega^{k(-n+1)} a \omega^i x 
              = \omega^{i+k(-n+1)} ax.\]
Lastly, \cite[Lemma 5.8]{DEi17} shows that the eigenvalues of $(1,\sigma)$ on $\Delta_1$ are $\{ \omega^{-2j+1} \colon j=1,\ldots,\defell \}$, hence the 
final claim follows from $\tau=\sigma^k$.             
\end{proof}

The following proposition summarises our `local-to-global' principle for skeleton groups: every skeleton group at depth $e$ in $\S_p(n)$ can be defined as $C_{n,e}(z)$ where $z\in\Omega_{n,j}$ is some `global fixed point' of $(\omega^{-j},\tau)$ for some~$j$.

\begin{proposition}
\label{equiv2}
If $G$ has depth $e$ in $\S_p(n)$, then the following are equivalent:
\begin{ithm}
\item $G \in \S_p(n)$ with $h \mid \gal(G)$;
\item $G \cong C_{n,e}(x)$ for some  $x \in \Delta_n$ and $u \in \U$ with $(u, \tau)(x) \equiv x \bmod \Gamma_{n+e}$;
\item $G \cong C_{n,e}(y)$ for some  $y \in \Delta_n$ and $v \in \ker(\chi)$ with $(v, \tau)(y) \equiv y \bmod \Gamma_{n+e}$;
\item $G \cong C_{n,e}(z)$ for some  $j \in \{ k(n-2i) \bmod d : i=1,\ldots,\defell \}$  and $z\in \Omega_{n,j}$.
\end{ithm}
\end{proposition}

\begin{proof}
Parts a,b,c) are Theorem~\ref{isomautpr}c) and Lemma \ref{equiv1} and listed for completeness only. It is trivially true that d) implies b), 
and it remains to show that c) implies d); we proceed with an argument similar to the proof of \cite[Lemma 5.6]{DEi17}.  In the following 
suppose that $G \cong C_{n,e}(y)$ for some  $y \in \Delta_n$ with  $(v, \tau)(y) \equiv y \bmod \Gamma_{n+e}$ for a suitable $v \in \ker(\chi)$. 
Lemma \ref{lattice} implies that we can decompose 
$y = y_{i_1} + \ldots + y_{i_h}$ with each $y_i \in \Sigma_{n,i}$. Write $v = \omega^j \theta^q u$ for some $j,q \in \Z$ and  $u \in \U_2$, 
so that each $\rho_s(v) = \omega^{j} \rho_s(u)$ with $\rho_s(u) \in \U_2$. In the following let $i\in\{i_1,\ldots,i_h\}$. Since $v\in\ker(\chi)$, that is, $\tau(v)=v$, we have
\[
(1,\tau)(v,1)(y_i)= (v, \tau)(y_i)= (v,1)(1,\tau)(y_i)= \omega^i(v,1)(y_i) = \omega^{i-j} (u,1)(y_i).
\]
Note that $(v,1)(y_i)\in \Sigma_{n,i}$, hence $(v,\tau)(y_i)\in \Sigma_{n,i}$, so  the assumption $y-(v,\tau)(y)\in\Gamma_{n+e}$ together with Lemma \ref{lattice} implies that for every $i$ 
we have $y_i-(v,\tau)(y_i)=y_i-\omega^{i-j}(u,1)(y_i)\in \Gamma_{n+e}$, that is,\[(u,1)(y_i)\equiv \omega^{j-i}y_i\bmod \Gamma_{n+e}.\]
This shows that $(u^{\defd},1)(y_i)\equiv y_i\bmod \Gamma_{n+e}$ and, since the action of $(u,1)$ on $\Gamma_1/\Gamma_{n+e}$ has $p$-power order, 
it follows that there is some $m$ such that $(u^{\defd},1)^m$ and $(u,1)$ induce the same action. The former acts trivially on $y_i$ while the 
latter acts by multiplication by $\omega^{j-i}$. Thus, if $\omega^{j-i}\ne 1$, then this forces $y_i\in\Gamma_{n+e}$. In conclusion, 
$y\equiv y_j\bmod \Gamma_{n+e}$, and since $y \in \Delta_n$ we also have $y_j \in \Delta_n$. Thus, $G\cong C_{n,e}(y)=C_{n,e}(y_j)$, and we can take $z=y_j$; 
the claim follows with Lemma \ref{xiny}.
\end{proof}

Theorem \ref{thmGFP} is an immediate consequence of Proposition \ref{equiv2} and allows us to prove Theorem \ref{skelbottom}.

\enlargethispage{2ex}

\begin{proof}[Proof of Theorem \ref{skelbottom}]
  Let $\R$ be a Galois tree in $\S_p(n)$ with Galois order $h$. Recall that $\dep(\S_p(n))=n-\defc$. Let $G$ be a group in $\R$ at depth $e$ in $\S_p(n)$. By Theorem \ref{thmGFP}, we have $G\cong C_{n,e}(x)$ for 
  some $x\in\Delta_n$ with $(1,\tau)(x)=\omega^i x$ for some $i$. The group $H=C_{n,n-\defc}(x)$ is a descendant of $G$ in $\S_p(n)$, hence $\gal(H)$ divides $\gal(G)$. On the other hand, 
  Proposition \ref{equiv2} shows that $\gal(G)$ divides $\gal(H)$. Thus, $H$ is a leaf in $\R$, and every leaf of $\R$ has depth $\dep(\S_p(n))$.
\end{proof}

\section{Ramification of Galois trees}
\label{isorev}
\noindent To describe the structure of Galois trees in more detail, it is required to solve the isomorphism problem of skeleton groups defined by global Galois fixed points; 
see Theorem \ref{thmGFP}. The aim of this section is to give a generalisation of \cite[Lemma 5.11]{DEi17}, and to indicate why it appears to be a lot more difficult to prove 
an explicit structure result for Galois trees that generalises the results for $h=\defd$ in \cite[Theorem~6.2]{DEi17}. We start with a preliminary lemma; as before, $\tau=\sigma^k$ with $d=p-1=hk$.

\begin{lemma}\label{lemFP}
Let $x,y\in\Sigma_{n,i}$ such that $(u,1)(x)\equiv y\bmod \Gamma_{n+e}$ for some $u\in\U$. If $u=vw$ with $v\in\ker(\chi)$ and $w\in\im(\chi)$, 
then $(v,1)(x)\equiv y\bmod \Gamma_{n+e}$. 
\end{lemma} 
\begin{proof}
Since $y$ and $(u,1)(x)$ lie in $\Sigma_{n,i}$, it follows from $y\equiv(u,1)(x)\bmod \Gamma_{n+e}$ that
\[\omega^i y\equiv \omega^i (\tau(u),1)(x)\equiv \omega^i (\chi(u)^{-1},1)(u,1)(x)\equiv \omega^i(\chi(u)^{-1},1)(y)\bmod \Gamma_{n+e},\]
hence $(\chi(u),1)$ fixes $y$ modulo $\Gamma_{n+e}$ and, by symmetry, also $x$ modulo $\Gamma_{n+e}$. Note that $\chi(u)=\chi(w)$, so also $(\chi(w),1)$ 
fixes both $x$ and $y$ modulo $\Gamma_{n+e}$. Now iterate the argument: since $(\chi(w),1)$ fixes $x$ modulo $\Gamma_{n+e}$, it follows that $(\chi^2(w),1)$ 
fixes $x$ modulo $\Gamma_{n+e}$ and, by induction, $(\chi^m(w),1)$ does the same  for all $m\in \N$. Let $j$ be large enough such that 
$\U_j=1+\p^j$ acts trivially on $\Gamma_n/\Gamma_{n+e}$. Similar to the decomposition of $\U$, we have $\U_j=\ker(\chi|_{\U_j})\times \im(\chi|_{\U_j})$. 
Since $\U=\ker(\chi) \times \im(\chi)$, it follows that $\chi$ induces an automorphism of $J=\im(\chi)/\im(\chi|_{\U_j})$, denoted $\psi$. Note that $J$ is finite since 
it is a quotient of $\im(\chi)/\im(\chi)^{p^r}$ for some $r\geq 1$ and the latter group is finite because it is finitely generated abelian with finite exponent.
It follows that $\psi$ has finite order $m\in\N$. This shows that $\psi^m(\hat{w})=\hat{w}$ for the coset $\hat{w}=w\im(\chi|_{\U_j})\in J$,  that is, $\chi^m(w)=ws$ 
for some $s\in \U_j$. Since $(s,1)$ and $(\chi^m(w),1)=(ws,1)$ fix $x$ modulo $\Gamma_{n+e}$, it follows that $(w,1)$ fixes $x$ modulo $\Gamma_{n+e}$. 
The claim follows.
\end{proof}

\enlargethispage{4ex}
\begin{proposition} \label{propIsoGT}
Let $x\in\Omega_{n,i}$ and $y\in\Omega_{n,j}$. Then $C_{n,e}(x) \cong C_{n,e}(y)$ if and only if $i = j$ and $(v,\sigma^m)(x)\equiv y \bmod \Gamma_{n+e}$ for some $v\in \ker(\chi)$ and $m\in\{0,\ldots,k-1\}$.
\end{proposition}
\begin{proof} 
By Theorem~\ref{isomautpr}b) the groups are isomorphic, if and only if there exists $(u,\varphi)\in \U\split \G$ with \[(u,\varphi)(x)\equiv y\bmod \Gamma_{n+e}.\]
Now observe that modulo $\Gamma_{n+e}$, we have
\begin{eqnarray*}\omega^j y &\equiv & (1,\tau)(y)\equiv (1,\tau)(u,\varphi)(x) \\&\equiv&\omega^i (\tau(u),\varphi)(x)\equiv \omega^i (\chi(u)^{-1},1)(u,\varphi)(x)\\&\equiv& \omega^i (\chi(u)^{-1},1)(y),
\end{eqnarray*}
that is, $(\chi(u)^{-1},1)(y)\equiv \omega^{j-i} y \bmod \Gamma_{n+e}$. Since the image of $\chi$ is contained in $\langle\theta\rangle\times \U_2$, the action on $\Gamma_1/\Gamma_{n+e}$ 
induced by $(\chi(u)^{-1},1)$ has $p$-power order. This forces $\omega^{j-i}=1$, so 
$j=i$, as claimed.

Recall that $\tau=\sigma^k$ and write $\varphi=\sigma^r$ with $r=qk+m$ where $m\in\{0,\ldots,k-1\}$. This shows that $(1,\varphi)(x)=(1,\sigma^m)(1,\tau^q)(x)=\omega^{qi}(1,\sigma^m)(x)$, 
which proves that 
\[y\equiv (u,\varphi)(x)\equiv \omega^{qi}(u,1)(1,\sigma^m)(x)\equiv (u\omega^{-qi},\sigma^m)(x)\bmod \Gamma_{n+e}.\]
Thus, without loss of generality, $\varphi=\sigma^m$ with $m\in\{0,\ldots,k-1\}$. Note that $(1,\sigma^m)(x) \in \Sigma_{n,i}$, and  $(u,1)$ maps $(1,\sigma^m)(x)$ 
to $y$ modulo $\Gamma_{n+e}$.  Writing $u=vw$ with $v\in\ker(\chi)$ and $w\in\im(\chi)$, 
it follows from Lemma \ref{lemFP} that $(v,1)$ maps $(1,\sigma^m)(x)$ to $y$ modulo $\Gamma_{n+e}$. The claim follows.
\end{proof}

\begin{corollary}\label{corSamey} 
If $\mathcal{R}$ is a Galois tree in $\S_p(n)$ with Galois order $h$, then all groups in $\mathcal{R}$ can be defined by elements in $\Omega_{n,i}$ for a unique  
$i\in \{k(n-2j) \bmod \defd: j\in\{1,\ldots,\defell\}\}$.
\end{corollary}

\begin{proof}
Let $G$ be the root of $\mathcal{R}$ and let $H$ be an $m$-step descendant of $G$ in $\mathcal{R}$. By Theorem \ref{thmGFP}, we can define 
$G=C_{n,e}(x)$ and $H=C_{n,e+m}(y)$ for some $x\in \Omega_{n,i_1}$ and $y\in\Omega_{n,i_2}$. Since $H$ is a descendant of $G$, 
we have $C_{n,e}(x)\cong C_{n,e}(y)$, and now Proposition \ref{propIsoGT} implies that the claim is true for $i=i_1=i_2$. It follows 
from Lemma \ref{xiny} that $i \in \{k(n-2j) \bmod \defd: j \in \{1,\ldots,\defell\}\}$.
\end{proof}

We say a Galois tree $\mathcal{R}$ in $\S_p(n)$ has \emph{type} $t$ if the groups in $\mathcal{R}$ are all defined by elements in $\Omega_{n,t}$, cf.\ Corollary \ref{corSamey}. 
The next lemma gives a necessary condition for $e$ to be a ramification level; the example below illustrates that this condition is not sufficient.

\begin{lemma}\label{lemRam} 
Let $\mathcal{R}$ be a Galois tree in $\S_p(n)$ of Galois order $h$ and type $t$. If $\mathcal{R}$ has ramification level $e$, 
then \[e\bmod h \in \{ 2(j-i)\bmod h : i \in \mathcal{I}_t, j\in\{1,\ldots,\defell\} \},\]
where $\mathcal{I}_t=\{i\in\{1,\ldots,\defell\} : t\equiv k(n-2i)\bmod \defd\}$.  
\end{lemma}

\begin{proof} 
\noindent Let $G_1$ and $G_2$ be distinct groups in $\mathcal{R}$ with the same immediate ancestor in $\mathcal{R}$; suppose $G_1$ and $G_2$ have depth $e+1$ in $\S_p(n)$. 
Theorem \ref{thmGFP} and Corollary \ref{corSamey} show that we can define each $G_i=C_{n,e+1}(x_i)$ for some $x_i\in \Omega_{n,t}$. Since $C_{n,e}(x_1)\cong C_{n,e}(x_2)$ 
is the isomorphism type of the parent of $G_1$ and $G_2$, Proposition \ref{propIsoGT} shows  that there is $v\in\ker(\chi)$ and $m\in\{0,\ldots,k-1\}$ with $(v,\sigma^m)(x_1)\equiv x_2\bmod \Gamma_{n+e}$, 
that is, $(v,\sigma^m)(x_1)=x_2 + y$ for some $y\in\Gamma_{n+e}$.  It follows from the definition of $\Sigma_{n,t}$ and the definition of $\chi$ that $(v,\sigma^m)(x_1)\in \Sigma_{n,t}$, 
and so $y=(v,\sigma^m)(x_1)-x_2 \in \Sigma_{n+e,t}$. Since $x_1$ and $(v,\sigma^m)(x_1)$ define isomorphic groups, we can assume, without loss of generality, that $x_1=x_2+y$ with $x_2\in\Omega_{n,t}$ 
and $y\in\Sigma_{n+e,t}$. Since $x_1$ and $x_2$ define non-isomorphic groups at depth $e+1$, we have in fact $y\in\Omega_{n+e,t}$. Since $x_1\in\Omega_{n,t}$, 
Lemma \ref{xiny} shows that  $t\equiv k(n-2i)\bmod \defd$ for some $i\in\{1,\ldots,\defell\}$.  Similarly, $y\in\Omega_{n+e,t}$ implies that $t\equiv k(n+e-2j)\bmod \defd$ 
for some $j\in\{1,\ldots,\defell\}$. Together, $h$ must divide $2(j-i)-e$, so  $e\equiv 2(j-i)\bmod h$, as claimed.
\end{proof}
 
\begin{example} Let  $p=7$ and $n=24$, and consider $h=6$ and $k=1$. By Lemma \ref{xiny} the possible types of Galois trees are $t\in\{2,4\}$.
 If $t=2$, then $j=2$, and a ramification level $e$ is only possible if $e\bmod 6 \in \{0,4\}$. If $t=4$, then $j=1$, and a ramification level $e$ can only occur for $e\bmod 6 \in \{0,2\}$.  
 However, \cite[Theorem~1]{DEi17} shows that the actual ramification levels are at depth $e$ with $e\bmod 6=4$ and $e\bmod 6=2$, respectively, see also Figure \ref{fignondistreg}.
\end{example}

\enlargethispage{4ex}

We now attempt to describe, up to isomorphism, the immediate descendants of a group in a Galois tree. For this we require the following set-up. Let $\mathcal{R}$ be a Galois tree in $\S_p(n)$ 
with Galois order $h$ and type $t$. Let $G=C_{n,e}(x)$ with $x\in\Omega_{n,t}$ be a group in $\mathcal{R}$, and assume that $e < \dep(\S_p(n))$. The proof of Lemma~\ref{lemRam} shows 
that the immediate descendants of $G$ in $\mathcal{R}$ are $C_{n,e+1}(x+y)$ where $y\in\Sigma_{n+e,t}$. Indeed, for each such $y$, the group $H=C_{n,e+1}(x+y)$ is an 
immediate descendant of $G$ and so $\gal(H)$ divides $h$; on the other hand, $x+y\in \Sigma_{n,t}$, so $h\mid \gal(H)$, hence $\gal(H)=h$ and $H\in \mathcal{R}$. If $y,z\in\Sigma_{n+e,t}$, 
then Proposition \ref{propIsoGT} shows that $C_{n,e+1}(x+y)\cong  C_{n,e+1}(x+z)$ if and only if there is some\footnote{Proposition \ref{propIsoGT} states $v\in\ker(\chi)$, but we can extend this 
to $\U$ by Theorem~\ref{isomautpr}; this is useful for applying  
\cite[Theorem~5.1]{Die10}.} $v\in\U$ and $m\in\{0,\ldots,k-1\}$ with $(v,\sigma^m)(x+y)\equiv x+z\bmod \Gamma_{n+e+1}$.  Since $y,z\in\Sigma_{n+e,t}$, the latter implies 
$(v,\sigma^m)(x)\equiv x\bmod \Gamma_{n+e}$, which, together with Theorem~\ref{isomautpr}c), shows that $|\sigma^m|$ divides $h=|\sigma^k|$, the Galois order of $G$. Thus, we 
have $\sigma^m\in \langle \sigma^k\rangle$, which forces $k\mid m$, and as $m\in \{0,\ldots,k-1\}$ we conclude that $m=0$. Thus, $C_{n,e+1}(x+y)\cong  C_{n,e+1}(x+z)$ if and only if 
there is some $v\in{\rm Stab}_{\U}(x+\Gamma_{n+e})$ with $(v,1)(x)-x +(v,1)(y)\equiv z\bmod \Gamma_{n+e+1}$. Motivated by this, we define 
\[\mathcal{Z}(n,e,t)=\{y + \Gamma_{n+e+1} : y\in\Sigma_{n+e,t} \} \cong \Sigma_{n+e,t}/\Sigma_{n+e+1,t}\]
and let $u\in {\rm Stab}_{\U}(x+\Gamma_{n+e})$ act affinely on $y+\Gamma_{n+e+1}\in \Gamma_{n+e}/\Gamma_{n+e+1}$ via
\begin{eqnarray}\label{eqaffine} u\cdot (y+\Gamma_{n+e+1})= (u,1)(x)-x + (u,1)(y)+\Gamma_{n+e+1}.
\end{eqnarray} It follows that the immediate descendants of $G$ are, up to isomorphism,
parametrised by the affine\linebreak ${\rm Stab}_{\U}(x+\Gamma_{n+e})$-orbits in $\mathcal{Z}(n,e,t)$; however, it is important to note that $\mathcal{Z}(n,e,t)$ is in general not closed under the affine action 
of ${\rm Stab}_{\U}(x+\Gamma_{n+e})$, see item (1) below. We also stress that in the following all stabilisers are with respect to the normal non-affine action, that is, ${\rm Stab}_\U(x+\Gamma_{n+e})$ 
is the set of all $u\in \U$ such that $(u,1)(x)\equiv x\bmod \Gamma_{n+e}$. 

We now inspect the affine action in more detail. Let  $z=y+\Gamma_{n+e+1}\in \Gamma_{n+e}/\Gamma_{n+e+1}$ and let $u,v\in {\rm Stab}_{\U}(x+\Gamma_{n+e})$. A short calculation shows that 
$(uv)\cdot z=u\cdot(v\cdot z)$, and the following hold:
\begin{iprf}
\item[(1)] If $z\in\mathcal{Z}(n,e,t)$, then   $u\cdot z\in\mathcal{Z}(n,e,t)$ if and only if $\chi(u)\in{\rm Stab}_\U(x+y+\Gamma_{n+e+1})$.
\item[(2)] We have $u\cdot y\equiv v\cdot y\bmod \Gamma_{n+e+1}$ if and only if $uv^{-1}\in{\rm Stab}_{\U}(x+y+\Gamma_{n+e+1})$.
\end{iprf}

We summarise these results in a lemma; this proves Theorem \ref{thmGalprops}a).

\begin{lemma}\label{lemID}
Let $\mathcal{R}$ be a Galois tree in $\S_p(n)$ with Galois order $h$ and type $t$. Let $G=C_{n,e}(x)$ be a group in $\mathcal{R}$ with $x\in\Omega_{n,t}$ and $e < \dep(\S_p(n))$. 
If $\mathcal{M}\subseteq \Sigma_{n+e,t}$ defines a set of representatives of the affine ${\rm Stab}_{\U}(x+\Gamma_{n+e})$-orbits on $\mathcal{Z}(n,e,t)$, 
then the immediate descendants of $G$ in $\mathcal{R}$, up to isomorphism, are
\[\{C_{n,e+1}(x+y) : y\in \mathcal{M}\}.\]
If $p\equiv 5\bmod 6$, then there exists $e_0=e_0(p)$ such that for all $e_0 \leq e < n-c$ the number of immediate descendants of $G$ in $\mathcal{R}$ is a power of $p$.
\end{lemma}
\begin{proof}
 If $z\in\Sigma_{n,t}$, then a direct calculation shows that $(u,1)\in\U\rtimes \G$ stabilises $z+\Gamma_{n+e}$ if and only if $(\tau(u),1)$ does. Also, note that if $x$ and $y$ are as in the lemma, 
 then ${\rm Stab}_\U(x+y+\Gamma_{n+e+1})$ is a subgroup of ${\rm Stab}_\U(x+\Gamma_{n+e})$. A short calculation shows the following:
\begin{iprf}
\item[(3)] If  $u\in{\rm Stab}_\U(x+\Gamma_{n+e})$ and $v\in{\rm Stab}_\U(x+y+\Gamma_{n+e+1})$, then $(\chi(u),1)(x+y)\equiv x+y \bmod \Gamma_{n+e+1}$ if and only if  
$(\chi(uv),1)(x+y)\equiv x+y \bmod \Gamma_{n+e+1}$.
\end{iprf}
Since $\theta\in\U$  acts trivially on $\Gamma_1$ and $\omega\in\U$ acts nontrivially on each $\Gamma_i/\Gamma_{i+1}$, Properties (1)--(3) imply that the ${\rm Stab}_{\U}(x+\Gamma_{n+e})$-orbit 
of $y+\Gamma_{n+e+1}\in\mathcal{Z}(n,e,t)$ intersects $\mathcal{Z}(n,e,t)$ in a set of $p$-power size
\begin{eqnarray}\label{eqo1}\qquad |\{u{\rm Stab}_{\U_2}(x+y+\Gamma_{n+e+1}) : u\in {\rm Stab}_{\U_2}(x+\Gamma_{n+e}),\; \chi(u)\in {\rm Stab}_{\U_2}(x+y+\Gamma_{n+e+1})\}|.
\end{eqnarray}If $p\equiv 5\bmod 6$, 
then \cite[Theorem 5.1]{Die10} shows that (for sufficiently large $e$) we have
\begin{equation}\label{eqstab} \begin{array}{cc}{\rm Stab}_{\U_2}(x+\Gamma_{n+e+\defd})= ({\rm Stab}_{\U_2}(x+\Gamma_{n+e}))^{[p]}\quad\text{and}\\[1ex] {\rm Stab}_{\U_2}(x+\Gamma_{n+e})\leq {\rm Stab}_{\U_2}(\Gamma_i/\Gamma_{i+3\defd})\text{ for all $i$};\end{array}
\end{equation}
here $({\rm Stab}_{\U_2}(x+\Gamma_{n+e}))^{[p]}$ denotes the subgroup generated by all $p$-th powers. In this situation,\linebreak ${\rm Stab}_{\U_2}(x+\Gamma_{n+e})$ stabilises $\Gamma_{n+e}/\Gamma_{n+e+1}$, 
hence \eqref{eqo1} is independent of $y$ and each ${\rm Stab}_{\U}(x+\Gamma_{n+e})$-orbit in $\mathcal{Z}(n,e,t)$ has size
\begin{eqnarray}\label{eqo}\qquad o=|\{u{\rm Stab}_{\U_2}(x+\Gamma_{n+e+1}) : u\in {\rm Stab}_{\U_2}(x+\Gamma_{n+e}),\; \chi(u)\in {\rm Stab}_{\U_2}(x+\Gamma_{n+e+1})\}|.
\end{eqnarray} 
Thus, if $p\equiv 5\bmod 6$, then the number of affine orbits is $|\mathcal{Z}(n,e,t)|/o$. Since $\mathcal{Z}(n,e,t)$ is a section of the $p$-group $\Gamma_1/\Gamma_{n+e+1}$, it follows that 
$|\mathcal{Z}(n,e,t)|/o$ is a power of $p$, as claimed.
\end{proof}
  
This result implies that the ramification levels in a Galois tree occur with periodicity $\defd$, at least when $p\equiv 5\bmod 6$. We prove this in the following proposition, which
yields Theorem \ref{thmGalprops}b). The proof is similar to the proof of Theorem~\ref{thm_newperiod_short}; the difference is that in this proposition we only consider descendants in 
a Galois tree contained in $\S_p(n)$, whereas Theorem \ref{thm_newperiod_short} considers descendants in the branch $\B_p(n)$, which leads to additional difficulties.

\begin{proposition}\label{propDes5}
If $p\equiv 5\bmod 6$, then there exists $e_0=e_0(p)$, such that for all $e_0 \leq e < \dep(\S_p(n))-\defd$ the following holds.  If $\mathcal{R}$ is a Galois tree in $\S_p(n)$ with Galois order $h$ and type $t$, then $\mathcal{R}$ 
has a ramification level $e$ if and only if $\mathcal{R}$ has a ramification level $e+\defd$. More precisely, if $x \in \Omega_{n,t}$, then $C_{n,e+\defd}(x)$
has exactly $m$ immediate descendants in $\R$ at depth $e+\defd+1$ in $\S_p(n)$ if and only if $C_{n,e}(x)$ has exactly $m$ immediate descendants in $\R$ at depth $e+1$ in $\S_p(n)$.
\end{proposition}

\begin{proof}
We use Lemma \ref{lemID} to describe the immediate descendants of $C_{n,e}(x)$ and $C_{n,e+d}(x)$ in $\mathcal{R}$. Note that multiplication by $p$ induces a bijection \[\phi\colon \mathcal{Z}(n,e,t)\to \mathcal{Z}(n,e+\defd,t).\] 
As in the proof of Lemma \ref{lemID}, the assumption $p\equiv 5\bmod 6$ implies that for sufficiently large $e$ we have
\[ {\rm Stab}_{\U_2}(x+\Gamma_{n+e+\defd})= ({\rm Stab}_{\U_2}(x+\Gamma_{n+e}))^{[p]}\] and every $u\in {\rm Stab}_{\U_2}(x+\Gamma_{n+\defd+e})$ acts trivially on $\Gamma_i/\Gamma_{i+3\defd}$ 
for every $i$,  see  \cite[Theorem 5.1]{Die10}. Now let $u\in {\rm Stab}_{\U_2}(x+\Gamma_{n+e})$. By assumption, $(u,1)(x)-x = y$ lies in $\Gamma_{n+e}$, 
so $(u^p,1)(x)-x \equiv py\bmod \Gamma_{n+e+3d}$. This proves that 
for $z+\Gamma_{n+e+1}\in\mathcal{Z}(n,e,t)$ we have
\begin{eqnarray*}\phi( u\cdot (z+\Gamma_{n+e+1})) &=& \phi( (u,1)(x)-x+(u,1)(z)+\Gamma_{n+e+1})\\&=&
  \phi( y+z+\Gamma_{n+e+1})\\&=& py+pz+\Gamma_{n+e+\defd+1}\\
   &=& (u^p,1)(x)-x+pz+\Gamma_{n+e+\defd+1}\\
   &=& u^p\cdot \phi(z+\Gamma_{n+e+1}).
\end{eqnarray*}
Note that $\theta$ acts trivially on  $\Gamma_1$; moreover, $\omega^p=\omega$, and ${\rm Stab}_{\U_2}(x+\Gamma_{n+e+\defd})= ({\rm Stab}_{\U_2}(x+\Gamma_{n+e}))^{[p]}$ 
by assumption. Thus, $\phi$ induces a bijection between the affine ${\rm Stab}_{\U}(x+\Gamma_{n+e})$-orbits in $\mathcal{Z}(n,e,t)$ and the affine  
${\rm Stab}_{\U}(x+\Gamma_{n+e+\defd})$-orbits in $\mathcal{Z}(n,e+\defd,t)$. Now the claim follows with Lemma \ref{lemID}.
 \end{proof}

%%%%%%%%%%%%%%%%%%%%%%%%%%%%%%%%%%%%%%%%%%%%%%%%%%%%%%%%%%%%%%%%%%%%%%%%%%%%%
\section{The roots of Galois trees}
\label{galroots}

\noindent In this section we investigate the structure of the roots of Galois trees. Throughout, we assume that $G$ is a candidate for a root 
of a Galois tree in $\S_p(n)$ with Galois order $h$ and type $t$. It will be convenient to describe $G$ as a descendant of its immediate ancestor, which 
is why we assume that $G$ has depth $e+1$ in $\S_p(n)$, so the parent has depth $e$. Suppose this parent has Galois order $b=ha$ with 
$a\geq 1$ a divisor of $k$; thus, $G$ is a root of a Galois tree if and only if $a\geq 2$. Throughout let $\nu\in \G$ such that \[\nu^a=\tau;\] we can assume 
that $\tau=\sigma^k$ and $\nu=\sigma^{k/a}$ with $d=hk=b(k/a)$. Recall that the sets $\Sigma_{n,i}$ and $\Omega_{n,i}$ have been defined with respect to our 
fixed generator $\tau\in \G$ of order $h$. We require the analogous definitions for $\nu$, that is, for $i \in \{1,\ldots,d-1\}$ we set
\[\Sigma_{n,i}'=\{x\in\Gamma_n : (1,\nu)(x)=\omega^i x\}\quad\text{and}\quad \Omega_{n,i}'=\{x\in\Delta_n :(1,\nu)(x)=\omega^i x\}\]
and define $\chi' \colon \U \ra \U$ via $\chi'(u)=u\nu(u)^{-1}$.

\begin{theorem}\label{thmRoot}
With the previous notation, the following holds. The group $G$ is the root of a Galois tree with immediate ancestor of Galois order $b=ha$ for some $a\geq 2$ if and only if
\[G\cong C_{n,e+1}(x+y)\]
for some  $x\in \Omega_{n,s}'$ and $y\in\Omega_{n+e,t}\setminus\Omega_{n+e,s}'$ such that $\omega^{sa}=\omega^t$ and $\gal(C_{n,e}(x))=b$ and 
\begin{eqnarray} \label{eqroot}
(\omega^{-s},\nu)[(u,1)(x+y)]\not\equiv (u,1)(x+y)\bmod \Gamma_{n+e+1}\quad \text{for all }u\in Q,
\end{eqnarray}
where $Q=\ker(\chi)\cap {\rm Stab}_{\U_2}(x+\Gamma_{n+e})$.
\end{theorem}
\enlargethispage{3ex}
Before we prove the theorem, we note that Condition \eqref{eqroot} can be replaced by 
\begin{eqnarray}\label{eqrootNEW}(\omega^{-s},\nu)[(u,1)(x+y)]\not\equiv (u,1)(x+y)\bmod \Gamma_{n+e+1}\quad \text{for all }u\in \mathcal{Q}, \end{eqnarray}
where $\mathcal{Q}$ is a set of coset representatives of $\ker(\chi') \cap {\rm Stab}_{\U_2}(x+\Gamma_{n+e})$ in $Q$. 
This follows because if  $u\in\ker(\chi)$ satisfies $(\omega^{-s},\nu)[(u,1)(x+y)]\equiv (u,1)(x+y)\bmod \Gamma_{n+e+1}$, then $uv$ with $v\in\ker(\chi')$ also satisfies
\[(\omega^{-s},\nu)(uv,1)(x+y)\equiv(v,1)(\omega^{-s},\nu)(u,1)(x+y)\equiv (v,1)(u,1)(x+y)\bmod\Gamma_{n+e+1}.\] 
Moreover, Condition \eqref{eqroot} with $u=1$ implies that $(1,\nu)(y)\not\equiv \omega^{s}y\bmod\Gamma_{n+e+1}$; in particular, $y\notin \Omega_{n+e,s}'$.

\begin{proof}[Proof of Theorem \ref{thmRoot}]
We split the proof in two parts. \medskip
  
\noindent ``$\Longrightarrow$'':\; By Theorem \ref{thmGFP} we can assume that $G=C_{n,e+1}(z)$ for some $z\in\Omega_{n,t}$. Since the immediate ancestor $C_{n,e}(z)$ 
of $G$ has Galois order $b$ by assumption, Theorem~\ref{isomautpr}c) shows that $(u,\nu)(z)\equiv z\bmod \Gamma_{n+e}$ for some $u\in\U$. Let $\omega^{i_1},\ldots,\omega^{i_a}$ 
be the $d$-th roots of unity with $\omega^{a i_j}=\omega^t$ and decompose
\[\Sigma_{n,t}=\Sigma_{n,i_1}'\oplus\ldots\oplus \Sigma_{n,i_a}'.\]
Write $z=z_{i_1}+\ldots+z_{i_a}$ with each $z_{i_j}\in \Sigma_{n,i_j}'$. Recall that $(1,\nu)(z)\equiv (u^{-1},1)(z)\bmod \Gamma_{n+e}$, 
which shows  that $(u^{-d},1)(z)\equiv z\bmod \Gamma_{n+e}$ and 
\begin{eqnarray}\label{eqomega}
(\omega^{i_1}z_{i_1}+\ldots+\omega^{i_a}z_{i_a})\equiv (u^{-1},1)(z_{i_1}+\ldots+z_{i_a})\equiv \Gamma_{n+e}.
\end{eqnarray}
Write $u=\omega^s \theta^q v$ with $v\in\U_2$ and note that $z\equiv (u^{-d},1)(z)\equiv (\theta^{qd}v^{d},1)^{-1}(z)\bmod \Gamma_{n+e}$. Since $(\theta^q v,1)$ acts as an element 
of $p$-power order on $\Gamma_1/\Gamma_{n+e}$, this implies that $(\theta^q v,1)(z)\equiv z\bmod\Gamma_{n+e}$, and therefore $(u^{-1},1)(z)\equiv \omega^s z\bmod\Gamma_{n+e}$. 
Together with \eqref{eqomega}, it follows that
\[\omega^{i_1-s}z_{i_1}+\ldots+\omega^{i_a-s}z_{i_a} \equiv z_{i_1}+\ldots+z_{i_a}\bmod \Gamma_{n+e},\]
and with Lemma \ref{lattice} we deduce that each $(\omega^{i_j-s}-1)z_{i_j}\in\Gamma_{n+e}$. Since $(\omega^{i_j-s}-1)$ is a unit if and only if $s\ne i_j$, we have $z_{i_j}\in\Gamma_{n+e}$ 
if and only if $s\ne i_j$. This shows that
\[ z\equiv z_s \bmod \Gamma_{n+e},\]
and we can write $z=z_s + y$ for some $y\in\Gamma_{n+e}$. Recall that  $\tau=\nu^a$, so $(1,\tau)(z)=\omega^t z$  shows that
\[\omega^t z_s + \omega^t y = \omega^t z = (1,\nu^a)(z_s)+(1,\tau)(y)= \omega^{sa}z_s + (1,\tau)(y).\]
From $\omega^t=\omega^{sa}$ we deduce that \[\omega^t y =(1,\tau)(y),\]
hence $y\in \Sigma_{n+e,t}$. If we set $x=z_s$, then $G \cong C_{n,e+1}(x+y)$ as claimed. It remains to show that $y\in \Omega_{n+e,t}\setminus \Omega_{n+e,s}'$ and that \eqref{eqroot} holds: 
if there is $u\in\ker(\chi)$ such that $(\omega^{-s},\nu)[(u,1)(x+y)]\equiv (u,1)(x+y)\bmod \Gamma_{n+e+1}$, then Theorem~\ref{isomautpr}c) shows that $G\cong C_{n,e+1}((u,1)(x+y))$ 
has Galois order divisible by $b$. This is a contradiction to our assumption, hence  \eqref{eqroot} holds. Condition \eqref{eqroot} with $u=1$ implies that 
$(1,\nu)(y)\not\equiv \omega^{s}y\bmod\Gamma_{n+e+1}$, hence $y\notin\Omega'_{n+e,s}$. Moreover, we must have $y\in\Omega_{n+e,t}$ since $C_{n,e+1}(x+y)\not\cong C_{n,e+1}(x)$; 
recall that the groups have different Galois orders.

\medskip

\noindent ``$\Longleftarrow$'':\; Let $G \cong C_{n,e+1}(x+y)$ as in the theorem. The immediate ancestor of $G$ is $P=C_{n,e}(x)$ and has Galois order $b$ by assumption. 
Since $x+y\in\Sigma_{n,t}$, the Galois order of $G$ is divisible by $h$. Suppose, for a contradiction, that $G$ is not the root of a Galois tree. Then $G$ is a descendant 
of $P$ that lies in the same Galois tree as $P$, and Theorem~\ref{lemID} shows that $G\cong C_{n,e+1}(x+z)$ for some $z\in\Omega_{n+e,s}'$. By Theorem~\ref{isomautpr}b), 
there exists $(u^{-1},\sigma^j)\in\U\rtimes \G$ with \[(u^{-1},\sigma^j)(x+z)\equiv x+y\bmod \Gamma_{n+e+1}.\] Since $x+z\in \Sigma_{n,s}'$ is fixed by $(\omega^{-s},\nu)$, 
an argument as in the proof of Proposition \ref{propIsoGT} shows that we can assume $j\in\{0,\ldots,k/a-1\}$; recall that $\nu=\sigma^{k/a}$. If $j\ne 0$, then $(u^{-1},\sigma^j)(x)\equiv x\bmod \Gamma_{n+e}$, 
and Theorem~\ref{isomautpr}c) implies that $\gal(P)$ is  divisible by $|\sigma^j|>|\nu|$, which contradicts our assumption $\gal(P)=b$. 
Thus, $j=0$ and therefore \begin{eqnarray}\label{eqst}(u,1)(x+y)\equiv x+z\bmod \Gamma_{n+e+1},
\end{eqnarray}
which already implies that $u\in{\rm Stab}_{\U_2}(x+\Gamma_{n+e})$. Since $x+z$ and $x+y$ both lie in $\Sigma_{n,t}$,  Lemma \ref{lemFP} shows that we can also assume that  $u\in\ker(\chi)$. 
Now \eqref{eqst} contradicts our assumption \eqref{eqroot}, because  $x+z\in\Sigma_{n,s}'$ is fixed by $(\omega^{-s},\nu)$. This final contradiction shows that $G$ is the root of a Galois tree. 
\end{proof}

The next corollary is analogous to Lemma \ref{lemRam} and gives a necessary (but in general not sufficient) criterion for the depths of Galois tree roots with Galois order $h$. Recall our 
assumptions that $\tau=\sigma^k$, $\nu=\sigma^{k/a}$ and $\omega^{sa}=\omega^t$. 

\begin{corollary}
  Using the previous notation, a Galois tree root at depth $e+1$ with Galois order $h$ and immediate ancestor of Galois order $b=ha$ in a Galois tree of type $s$ can only exist if
  \[ e\equiv 2(j-i) \bmod h\]
 for some $i,j\in\{1,\ldots,\defell\}$ with $s\equiv (k/a)(n-2i)\bmod \defd$ and $t\equiv k(n+e-2j)\bmod \defd$.
\end{corollary}
\begin{proof}
  From $d=hk$ and $b=ha$ it follows that $d=b(k/a)$. 
  The groups in Theorem \ref{thmRoot} can only exist if $\Omega_{n,s}'\ne \emptyset \ne \Omega_{n+e,t}$, and  Lemma \ref{xiny} shows that this holds if and only if $s\equiv (k/a)(n-2i)\bmod \defd$ 
  and $t\equiv k(n+e-2j)\bmod \defd$ for some $i,j\in\{1,\ldots,\defell\}$. Since $sa\equiv t\bmod \defd$, the claim follows from
\[t\equiv k(n-2i)\equiv k(n+e-2j) \bmod \defd.\qedhere\]
\end{proof}

The following corollary concludes this section and proves  Theorem \ref{thmGalprops}c).

\begin{corollary}
 If $p\equiv 5\bmod 6$, then for large enough $e$ the following hold. If $C_{n,e+1}(x+y)$ with $e+\defd < \dep(\S_p(n))$ is the root of a Galois tree as in Theorem \ref{thmRoot}, then $C_{n,e+\defd+1}(x+py)$ 
is the root of a Galois tree.
\end{corollary}
\begin{proof}
  Write $G=C_{n,e+1}(x+y)$ and $\tilde G=C_{n,e+\defd+1}(x+py)$. The parent of $\tilde G$ is $\tilde P=C_{n,e+\defd}(x)$ with Galois order $b$. Suppose, for a contradiction, 
  that $\tilde G$ has Galois order $b$. By Theorem~\ref{lemID} we have 
  $\tilde G \cong C_{n,e+\defd+1}(x+pz)$ for some $z\in \Omega_{n+e,s}'$; recall that multiplication by $p$ induces an isomorphism $\Omega_{i,s}'\to\Omega_{i+\defd,s}'$ for all $i$.  
  Now Theorem~\ref{isomautpr}b) shows that there is $(u,\sigma^j)\in\U\rtimes \G$ with $(u,\sigma^j)(x+pz)\equiv x+py\bmod \Gamma_{n+e+\defd+1}$. Since $(u,\sigma^j)$ stabilises $x+\Gamma_{n+e+\defd}$, 
  we can assume, as in earlier proofs, that $j=0$ and so $u\in{\rm Stab}_{\U_2}(x+\Gamma_{n+e+\defd})$. Thus, $py\equiv (u,1)(x)-x+(u,1)(pz)\bmod \Gamma_{n+e+\defd+1}$, that is, $pz$ 
  and $py$ lie in the same orbit under the affine action defined in \eqref{eqaffine}. The proof of Proposition \ref{propDes5} shows that there is $v\in {\rm Stab}_{\U_2}(x+\Gamma_{n+e})$ 
  with $v^p=u$ such that  $y\equiv (v,1)(x)-x+(v,1)(z)\bmod \Gamma_{n+e+1}$, which implies that $G=C_{n,e+1}(x+y)\cong C_{n,e+1}(x+z)$. Since $x+z\in\Omega_{n,s}'$, 
  we have that $b$ divides $\gal(G)$, a contradiction. Thus, $\tilde G$ is a root of a Galois tree.
\end{proof}

%%%%%%%%%%%%%%%%%%%%%%%%%%%%%%%%%%%%%%%%%%%%%%%%%%%%%%%%%%%%%%%%%%%%%%%%%%%%%
\section{A new periodicity result}
\label{secperiod}\enlargethispage{4ex}
 
\noindent We now consider the structure of the pruned branches in $\T_p$ and proceed to prove our last main result. Recall from Section \ref{intro} that $\dep(\S_p(n))=n-\defc$, and $\B_p(n,n-\defc)\cong \B_p(n+\defd,n-\defc)$ 
for all large enough $n$; in order to describe the sequence of all pruned branches $\B_p(m,m-\defc)$ completely, it remains to specify the growth of the branches, that is,
the additional $d$ layers of groups in $\B_p(n+\defd,n+\defd-\defc)$ that are not covered by the isomorphism $\iota$. One way of doing that is to describe, for each skeleton 
group $G$ at depth $n-\defc$ in $\B_p(n+\defd,n+\defd-\defc)$, the $\defd$-step descendant tree $\mathcal{D}(G,\defd)$. It has been conjectured in \cite[Conjecture~1]{Die10} that for every such $G$ 
there exists a group $H$ at depth $n-\defd-\defc$ in $\B_p(n+\defd,n+\defd-\defc)$ such that $\mathcal{D}(G,\defd)\cong \mathcal{D}(H,\defd)$. Recall that in \cite{Die10} such a group $H$ is called a periodic 
parent of $G$, whereas we refer to such a group as a $\defd$-step twin.

A natural candidate for $H$ seems to be the $\defd$-step ancestor of $G$; indeed, it is shown in \cite[Theorem~1.2]{Die10} 
that if $p\equiv 5\bmod 6$ and if the $\defd$-step ancestor of $G$ has Galois order~$1$, then it is also a $\defd$-step twin. However, it is suggested in \cite[Remark 4]{Die10} and proved 
in \cite[Section 6.2.7]{saha} that the $\defd$-step ancestor is infinitely many times not a $\defd$-step twin, so it is still an open problem to describe the growth of the pruned branches.  
A crucial ingredient in the proof of \cite[Theorem 1.2]{Die10} is that if a skeleton group has Galois order~$1$, then the same holds 
for all its skeleton group descendants. In fact, \cite[Theorem~1.3]{Die10} considers a slight generalisation: informally, it proves that if a group $G$ (as above)  has distance at most 
$d$ from a certain path of groups with constant Galois order, then $G$ has a $d$-step twin. These results (together with \cite{DEi17}) highlight that the Galois order of skeleton groups 
seems to play a fundamental role for proving periodicity results concerning the growth of branches. The aim of this section is to exploit our new results on Galois trees to prove the 
following theorem, which yields a new periodicity result generalising \cite[Theorem~1.3]{Die10}; this proves \cite[Conjecture 1]{Die10} for $p\equiv 5\bmod 6$ and implies 
Theorem~\ref{thm_newperiod_short}.

\begin{theorem}\label{thm_newperiod}
Let $p \geq 7$ with $p\equiv 5\bmod 6$ and let $r$ be the total number of prime divisors of $\defd$. Then there exists $n_0=n_0(p)$ such that for all $n \geq n_0$ the following holds: For every skeleton group 
$G$ at depth $n-\defc$ in $\B_p(n+\defd)$ there is $m \leq pr - \defd$ such that the $m$-step ancestor $H$ of $G$ and the $\defd$-step ancestor $K$ of $H$ are $(m+\defd)$-step twins with the same Galois order; 
in particular, $G$ has a $\defd$-step twin at depth $n-\defd-\defc$ in~$\B_p(n+\defd)$.

\begin{figure}[ht!]
\begin{center}\includegraphics[scale=0.3]{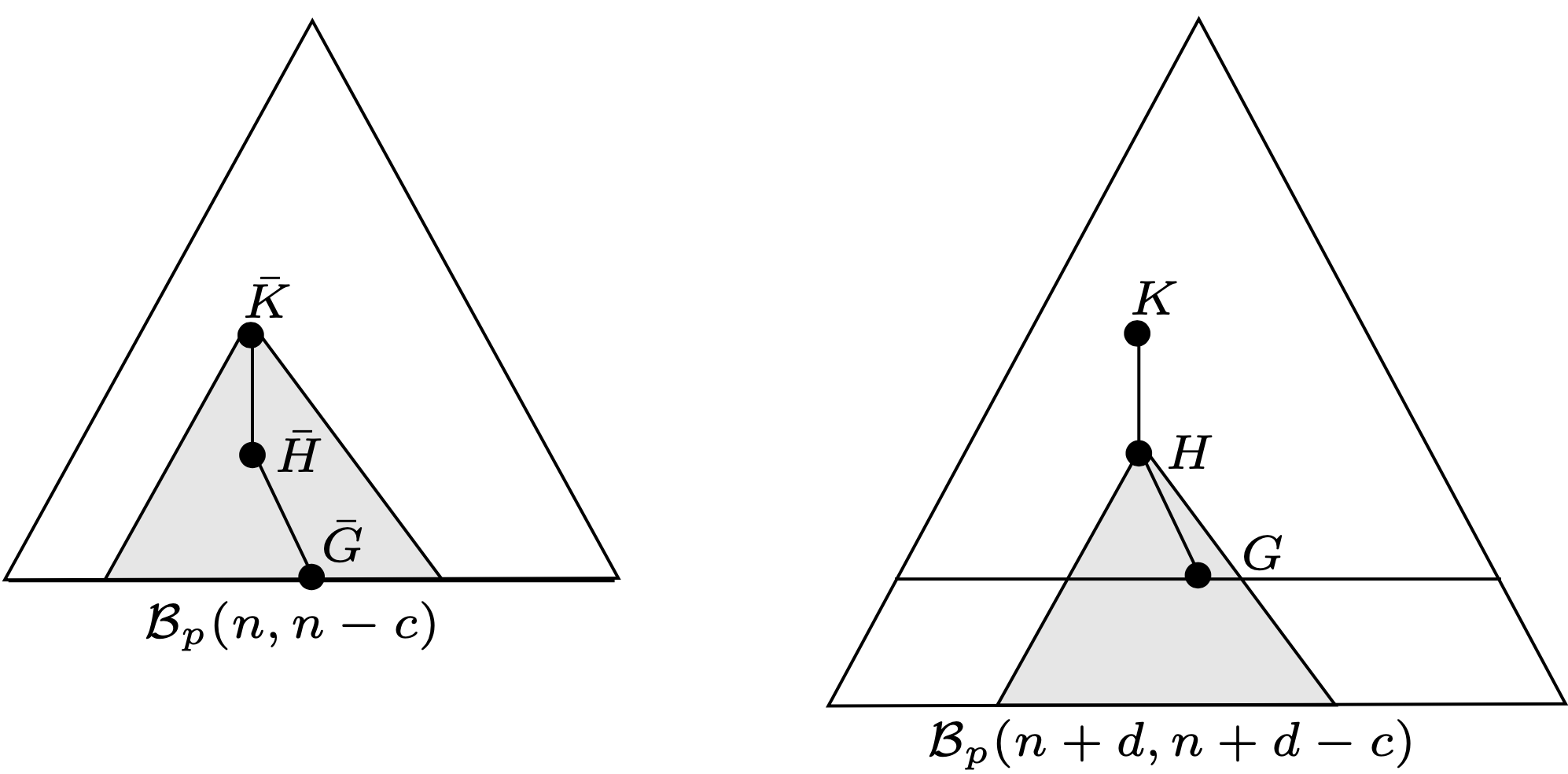}
\end{center}
  \caption{The groups $G,H,K$ and 
$\bar G,\bar H,\bar K$ correspond under the isomorphism $\B_p(n,n-\defc)\cong \B_p(n+\defd,n-\defc)$; 
the group $H$ is a $m$-step ancestor of $G$ for a certain $m\leq pr-\defd$, and $K$ is a $\defd$-step ancestor
of $H$. The $(m+\defd)$-step descendant trees of $K$, $H$ and $\bar K$ are isomorphic, which defines 
a twin for $G$.}\label{fig_newperiod} 
\end{figure}
\end{theorem}

\begin{proof}
 First we show that $m$ (as in the statement of the theorem) exists, and demonstrate that the last claim of the theorem follows from the first. Let $G$ be a skeleton group at depth $n-\defc$ in $\mathcal{B}_p(n+\defd,n+\defd-\defc)$. Since the Galois order of groups in $\B_p(n)$ can change at most $r$ times, it follows that any path of length $pr$ in $\B_p(n)$ has a 
subpath of length $\defd$ whose groups have the same Galois order. This proves that there exists $m\leq pr-\defd$ such that the $m$-step ancestor $H$ and $(m+\defd)$-step ancestor $K$ of $G$ have 
the same Galois order. We argue  below that $\mathcal{D}(H,m+\defd)\cong \mathcal{D}(K,m+\defd)$; if $P$ is the group in $\mathcal{D}(K,m+\defd)$ that corresponds to $G$ under this isomorphism 
of descendant trees, then $P$ is a $m$-step descendant of $K$ that is a $\defd$-step twin of $G$, that is, $\mathcal{D}(G,\defd)\cong \mathcal{D}(P,\defd)$. In particular, this proves that every skeleton 
group at depth $n-\defc$ in $\mathcal{B}_p(n+\defd,n+\defd-\defc)$ has a $\defd$-step twin, as claimed.
   
It remains to show that $\mathcal{D}(H,m+\defd)\cong \mathcal{D}(K,m+\defd)$. By assumption, $H$ and $K$ have the same Galois order, say $h$, and by Theorem \ref{thmGFP}, we can assume 
that $H=C_{n+\defd,n-\defc-m}(f)$ for some $f\in\Omega_{n+\defd,y}$ such that $(1,\tau)(f)=\omega^y f$ for $\tau=\sigma^{\defd/h}$. To simplify the notation, in the following we  write $\hatn=n+\defd$ and $e=n-\defc-m$, and we define $K_i=C_{\hatn,i}(f)$. Note that $S_\hatn\to K_1\to \ldots \to K_{\hatn-\defc}$ is a maximal path in $\B_p(\hatn,\hatn-\defc)$ 
with $K=K_{e-\defd}$ and $H=K_e$. With this notation, it remains to show that \[\mathcal{D}(K_e,m+\defd)\cong \mathcal{D}(K_{e-\defd},m+\defd),\] and our assumption is that $K_{e-\defd},K_{e-\defd+1},\ldots, K_{e+m+\defd}$ 
all have the same Galois order $h$. This statement is proved in  \cite[Theorem 1.3]{Die10} for $m+\defd\leq 2\defd$, but the same proof works for $m+\defd\leq pr$: the crucial ingredient is \eqref{eqstab} which is formulated in \cite{Die10} for $3d$, but the proof of \eqref{eqstab} shows that the same result holds for $3pr$ if one increases the bound $e_0=e_0(p)$.
\end{proof}

\clearpage
\appendix
%%%%%%%%%%%%%%%%%%%%%%%%%%%%%%%%%%%%%%%%%%%%%%%%%%%%%%%%%%%%%%%%%%%%%%%%%%%%%
\section{Examples}
\label{examples}

\noindent We exhibit a few examples of skeletons which have been determined using the ANUPQ
package \cite{anupq} for GAP \cite{GAP4}. We use a compact form to draw trees: 
Each vertex is labeled with the Galois order of the corresponding group, and a number 
$m$ on the right of a vertex indicates that this vertex and its full descendant tree 
appears $m$ times in the graph. 

\begin{figure}[ht!]
\begin{center}
\hfill
\begin{tikzpicture}[xscale=1,yscale=0.75]
    % First we draw the vertices
    \foreach \pos / \name / \typ / \gal in 
    {{(0,0)/v11/gvert/5}, 
    {(0,-1)/v21/gvert/5}, {(1,-1)/v22/gvert/5}, 
    {(0,-2)/v31/gvert/5}, {(1,-2)/v32/gvert/5},
    {(0,-3)/v41/ivert/}} % dummy for alignment
    \node[\typ] (\name) at \pos {\tiny \gal};
             
    \node[right=-0.075cm of v22] {\tiny 10};
    \node[right=-0.075cm of v31] {\tiny 11};
    
    % Connect vertices with edges and draw weights
    \foreach \source / \dest in 
    {v11/v21, v11/v22,
     v21/v31, v22/v32} 
%      v31/v41, v32/v42}
    {
        \path[edge] (\source) -- node[] {} (\dest);
    }
\end{tikzpicture}
\hfill
\begin{tikzpicture}[xscale=1,yscale=0.75]
    % First we draw the vertices
    \foreach \pos / \name / \typ / \gal in 
    {{(0,0)/v11/gvert/1}, 
    {(0,-1)/v21/gvert/1}, {(2,-1)/v22/gvert/1}, {(3,-1)/v23/gvert/1}, 
    {(0,-2)/v31/gvert/1}, {(1,-2)/v32/gvert/1}, {(2,-2)/v33/gvert/1}, {(3,-2)/v34/gvert/1}, 
    {(0,-3)/v41/gvert/1}, {(1,-3)/v42/gvert/1}, {(2,-3)/v43/gvert/1}, {(3,-3)/v44/gvert/1}}
    \node[\typ] (\name) at \pos {\tiny \gal};
             
    \node[right=-0.075cm of v21] {\tiny 7};
    \node[right=-0.075cm of v22] {\tiny 5};
    \node[right=-0.075cm of v31] {\tiny 6};
    \node[right=-0.075cm of v32] {\tiny 7};
    \node[right=-0.075cm of v33] {\tiny 7};
    \node[right=-0.075cm of v34] {\tiny 7};
    \node[right=-0.075cm of v41] {\tiny 7};
    \node[right=-0.075cm of v43] {\tiny 7};
    \node[right=-0.075cm of v44] {\tiny 49};

    % Connect vertices with edges and draw weights
    \foreach \source / \dest in 
    {v11/v21, v11/v22, v11/v23,
     v21/v31, v21/v32, v22/v33, v23/v34,
     v31/v41, v32/v42, v33/v43, v34/v44}
    {
        \path[edge] (\source) -- node[] {} (\dest);
    }
\end{tikzpicture}
\hfill\hfill\hfill
\end{center}
\caption{A Galois tree in $\S_{11}(20)$ (left) and a Galois tree in $\S_7(11)$ (right)}
\label{fignondistreg}
\end{figure}

\begin{figure}[ht!] 
\begin{center}
\begin{tikzpicture}[xscale=1,yscale=0.75]

    % First we draw the vertices (including Gal.ord. and mult.)
    \foreach \pos / \name / \gal / \mult in
    {{(1,-1)/v-i1-j1/36/1}, 
     {(1,-2)/v-i2-j1/6/1}, {(9,-2)/v-i2-j2/6/1}, 
     {(1,-3)/v-i3-j1/6/1}, {(9,-3)/v-i3-j2/6/1}, 
     {(1,-4)/v-i4-j1/6/6}, {(5,-4)/v-i4-j2/6/1}, {(9,-4)/v-i4-j3/6/1}, 
     {(1,-5)/v-i5-j1/6/1}, {(5,-5)/v-i5-j2/6/1}, {(8,-5)/v-i5-j3/3/3}, {(9,-5)/v-i5-j4/6/1}, {(12,-5)/v-i5-j5/3/3}, 
     {(1,-6)/v-i6-j1/6/1}, {(5,-6)/v-i6-j2/6/1}, {(8,-6)/v-i6-j3/3/1}, {(9,-6)/v-i6-j4/6/7}, {(12,-6)/v-i6-j5/3/7}, 
     {(1,-7)/v-i7-j1/6/1}, {(4,-7)/v-i7-j2/3/3}, {(5,-7)/v-i7-j3/6/1}, {(8,-7)/v-i7-j4/3/1}, {(9,-7)/v-i7-j5/6/1}, {(12,-7)/v-i7-j6/3/1}, 
     {(1,-8)/v-i8-j1/6/1}, {(4,-8)/v-i8-j2/3/1}, {(5,-8)/v-i8-j3/6/1}, {(8,-8)/v-i8-j4/3/1}, {(9,-8)/v-i8-j5/6/1}, {(12,-8)/v-i8-j6/3/1}, 
     {(1,-9)/v-i9-j1/6/1}, {(4,-9)/v-i9-j2/3/1}, {(5,-9)/v-i9-j3/6/1}, {(8,-9)/v-i9-j4/3/1}, {(9,-9)/v-i9-j5/6/1}, {(12,-9)/v-i9-j6/3/1}, 
     {(1,-10)/v-i10-j1/6/7}, {(4,-10)/v-i10-j2/3/7}, {(5,-10)/v-i10-j3/6/7}, {(8,-10)/v-i10-j4/3/7}, {(9,-10)/v-i10-j5/6/1}, {(12,-10)/v-i10-j6/3/1}, 
     {(1,-11)/v-i11-j1/6/1}, {(4,-11)/v-i11-j2/3/1}, {(5,-11)/v-i11-j3/6/1}, {(7,-11)/v-i11-j4/3/3}, {(8,-11)/v-i11-j5/3/7}, {(9,-11)/v-i11-j6/6/1}, {(11,-11)/v-i11-j7/3/3}, {(12,-11)/v-i11-j8/3/7}, 
     {(1,-12)/v-i12-j1/6/1}, {(4,-12)/v-i12-j2/3/1}, {(5,-12)/v-i12-j3/6/1}, {(7,-12)/v-i12-j4/3/1}, {(8,-12)/v-i12-j5/3/1}, {(9,-12)/v-i12-j6/6/7}, {(11,-12)/v-i12-j7/3/7}, {(12,-12)/v-i12-j8/3/7}, 
     {(1,-13)/v-i13-j1/6/1}, {(3,-13)/v-i13-j2/3/3}, {(4,-13)/v-i13-j3/3/7}, {(5,-13)/v-i13-j4/6/1}, {(7,-13)/v-i13-j5/3/1}, {(8,-13)/v-i13-j6/3/1}, {(9,-13)/v-i13-j7/6/1}, {(11,-13)/v-i13-j8/3/1}, {(12,-13)/v-i13-j9/3/1}, 
     {(1,-14)/v-i14-j1/6/1}, {(3,-14)/v-i14-j2/3/1}, {(4,-14)/v-i14-j3/3/1}, {(5,-14)/v-i14-j4/6/1}, {(7,-14)/v-i14-j5/3/1}, {(8,-14)/v-i14-j6/3/1}, {(9,-14)/v-i14-j7/6/1}, {(11,-14)/v-i14-j8/3/1}, {(12,-14)/v-i14-j9/3/1}, 
     {(1,-15)/v-i15-j1/6/1}, {(3,-15)/v-i15-j2/3/1}, {(4,-15)/v-i15-j3/3/1}, {(5,-15)/v-i15-j4/6/1}, {(7,-15)/v-i15-j5/3/1}, {(8,-15)/v-i15-j6/3/1}, {(9,-15)/v-i15-j7/6/1}, {(11,-15)/v-i15-j8/3/1}, {(12,-15)/v-i15-j9/3/1}, 
     {(1,-16)/v-i16-j1/6/7}, {(3,-16)/v-i16-j2/3/7}, {(4,-16)/v-i16-j3/3/7}, {(5,-16)/v-i16-j4/6/7}, {(7,-16)/v-i16-j5/3/7}, {(8,-16)/v-i16-j6/3/7}, {(9,-16)/v-i16-j7/6/1}, {(11,-16)/v-i16-j8/3/1}, {(12,-16)/v-i16-j9/3/1}, 
     {(1,-17)/v-i17-j1/6/1}, {(3,-17)/v-i17-j2/3/1}, {(4,-17)/v-i17-j3/3/1}, {(5,-17)/v-i17-j4/6/1}, {(6,-17)/v-i17-j5/3/3}, {(7,-17)/v-i17-j6/3/7}, {(8,-17)/v-i17-j7/3/7}, {(9,-17)/v-i17-j8/6/1}, {(10,-17)/v-i17-j9/3/3}, {(11,-17)/v-i17-j10/3/7}, {(12,-17)/v-i17-j11/3/7}, 
     {(1,-18)/v-i18-j1/6/1}, {(3,-18)/v-i18-j2/3/1}, {(4,-18)/v-i18-j3/3/1}, {(5,-18)/v-i18-j4/6/1}, {(6,-18)/v-i18-j5/3/1}, {(7,-18)/v-i18-j6/3/1}, {(8,-18)/v-i18-j7/3/1}, {(9,-18)/v-i18-j8/6/7}, {(10,-18)/v-i18-j9/3/7}, {(11,-18)/v-i18-j10/3/7}, {(12,-18)/v-i18-j11/3/7}, 
     {(1,-19)/v-i19-j1/6/1}, {(2,-19)/v-i19-j2/3/3}, {(3,-19)/v-i19-j3/3/7}, {(4,-19)/v-i19-j4/3/7}, {(5,-19)/v-i19-j5/6/1}, {(6,-19)/v-i19-j6/3/1}, {(7,-19)/v-i19-j7/3/1}, {(8,-19)/v-i19-j8/3/1}, {(9,-19)/v-i19-j9/6/1}, {(10,-19)/v-i19-j10/3/1}, {(11,-19)/v-i19-j11/3/1}, {(12,-19)/v-i19-j12/3/1}}
    {
    \ifthenelse{\mult > 1}
        { \node[gvert] (\name) at \pos {\tiny$\gal$}; \node[right=-0.075cm of \name] {\tiny$\mult$}; }
        { \node[gvert] (\name) at \pos {\tiny$\gal$}; }
    }

    % Connect gvertices with edges
    \foreach \source / \dest in
    {v-i1-j1/v-i2-j1, v-i1-j1/v-i2-j2, 
     v-i2-j1/v-i3-j1, v-i2-j2/v-i3-j2, 
     v-i3-j1/v-i4-j1, v-i3-j1/v-i4-j2, v-i3-j2/v-i4-j3, 
     v-i4-j1/v-i5-j1, v-i4-j2/v-i5-j2, v-i4-j2/v-i5-j3, v-i4-j3/v-i5-j4, v-i4-j3/v-i5-j5, 
     v-i5-j1/v-i6-j1, v-i5-j2/v-i6-j2, v-i5-j3/v-i6-j3, v-i5-j4/v-i6-j4, v-i5-j5/v-i6-j5, 
     v-i6-j1/v-i7-j1, v-i6-j1/v-i7-j2, v-i6-j2/v-i7-j3, v-i6-j3/v-i7-j4, v-i6-j4/v-i7-j5, v-i6-j5/v-i7-j6, 
     v-i7-j1/v-i8-j1, v-i7-j2/v-i8-j2, v-i7-j3/v-i8-j3, v-i7-j4/v-i8-j4, v-i7-j5/v-i8-j5, v-i7-j6/v-i8-j6, 
     v-i8-j1/v-i9-j1, v-i8-j2/v-i9-j2, v-i8-j3/v-i9-j3, v-i8-j4/v-i9-j4, v-i8-j5/v-i9-j5, v-i8-j6/v-i9-j6, 
     v-i9-j1/v-i10-j1, v-i9-j2/v-i10-j2, v-i9-j3/v-i10-j3, v-i9-j4/v-i10-j4, v-i9-j5/v-i10-j5, v-i9-j6/v-i10-j6, 
     v-i10-j1/v-i11-j1, v-i10-j2/v-i11-j2, v-i10-j3/v-i11-j3, v-i10-j3/v-i11-j4, v-i10-j4/v-i11-j5, v-i10-j5/v-i11-j6, v-i10-j5/v-i11-j7, v-i10-j6/v-i11-j8, 
     v-i11-j1/v-i12-j1, v-i11-j2/v-i12-j2, v-i11-j3/v-i12-j3, v-i11-j4/v-i12-j4, v-i11-j5/v-i12-j5, v-i11-j6/v-i12-j6, v-i11-j7/v-i12-j7, v-i11-j8/v-i12-j8, 
     v-i12-j1/v-i13-j1, v-i12-j1/v-i13-j2, v-i12-j2/v-i13-j3, v-i12-j3/v-i13-j4, v-i12-j4/v-i13-j5, v-i12-j5/v-i13-j6, v-i12-j6/v-i13-j7, v-i12-j7/v-i13-j8, v-i12-j8/v-i13-j9, 
     v-i13-j1/v-i14-j1, v-i13-j2/v-i14-j2, v-i13-j3/v-i14-j3, v-i13-j4/v-i14-j4, v-i13-j5/v-i14-j5, v-i13-j6/v-i14-j6, v-i13-j7/v-i14-j7, v-i13-j8/v-i14-j8, v-i13-j9/v-i14-j9, 
     v-i14-j1/v-i15-j1, v-i14-j2/v-i15-j2, v-i14-j3/v-i15-j3, v-i14-j4/v-i15-j4, v-i14-j5/v-i15-j5, v-i14-j6/v-i15-j6, v-i14-j7/v-i15-j7, v-i14-j8/v-i15-j8, v-i14-j9/v-i15-j9, 
     v-i15-j1/v-i16-j1, v-i15-j2/v-i16-j2, v-i15-j3/v-i16-j3, v-i15-j4/v-i16-j4, v-i15-j5/v-i16-j5, v-i15-j6/v-i16-j6, v-i15-j7/v-i16-j7, v-i15-j8/v-i16-j8, v-i15-j9/v-i16-j9, 
     v-i16-j1/v-i17-j1, v-i16-j2/v-i17-j2, v-i16-j3/v-i17-j3, v-i16-j4/v-i17-j4, v-i16-j4/v-i17-j5, v-i16-j5/v-i17-j6, v-i16-j6/v-i17-j7, v-i16-j7/v-i17-j8, v-i16-j7/v-i17-j9, v-i16-j8/v-i17-j10, v-i16-j9/v-i17-j11, 
     v-i17-j1/v-i18-j1, v-i17-j2/v-i18-j2, v-i17-j3/v-i18-j3, v-i17-j4/v-i18-j4, v-i17-j5/v-i18-j5, v-i17-j6/v-i18-j6, v-i17-j7/v-i18-j7, v-i17-j8/v-i18-j8, v-i17-j9/v-i18-j9, v-i17-j10/v-i18-j10, v-i17-j11/v-i18-j11, 
     v-i18-j1/v-i19-j1, v-i18-j1/v-i19-j2, v-i18-j2/v-i19-j3, v-i18-j3/v-i19-j4, v-i18-j4/v-i19-j5, v-i18-j5/v-i19-j6, v-i18-j6/v-i19-j7, v-i18-j7/v-i19-j8, v-i18-j8/v-i19-j9, v-i18-j9/v-i19-j10, v-i18-j10/v-i19-j11, v-i18-j11/v-i19-j12}
    \path[edge] (\source) -- (\dest);
\end{tikzpicture}
\end{center}
\caption{Galois trees with Galois orders $3$ and $6$ in $\S_{7}(24)$}
\label{figs724}
\end{figure}

\section{Notation} \label{glossary}

{\small
\begin{minipage}[t]{7.5cm}
\begin{longtable}{@{} lp{6cm}}
    $p$ & a prime $\geq 7$ \\
    $\defd$ & $p-1$ \\
    $\defc$ & $2p-8$ \\
    $\defell$ & $(p-3)/2$ \\
    $h, k$ & fixed positive integers with $hk = d$ \\
    $\gal(G)$ & the Galois order of $G$ \\
    $\newG_p$ & the graph associated with the finite $p$-groups of maximal class \\
    $\mathcal{T}_p$ & the maximal infinite tree in $\newG_p$ \\
    $S_p(n)$ & the group of order $p^n$ on the unique maximal infinite path in $\mathcal{T}_p$ \\
    $\B_p(n)$ & the $n$-th branch of $\T_p$ \\
    $\B_p(n,m)$ & the subtree of $\B_p(n)$ consisting of all groups of depth at most $m$\\
    $\mathcal{D}(G)$ & the descendant tree of $G$ in $\B_p(n)$ \\
    $\mathcal{D}(G,m)$ & the $m$-step descendant tree of $G$ in $\B_p(n)$ \\
    $\S_p(n)$ & the full subtree of $\B_p(n)$ consisting of all capable groups of depth at most $n-\defc$ 
\end{longtable}
\end{minipage}\hspace*{0.8cm}\begin{minipage}[t]{7.5cm}\begin{longtable}{@{} lp{6cm}}
     $\omega$ & a primitive $d$-th root of unity in $\Z_p$ \\
     $\theta$ & a primitive $p$-th root of unity in $\Q_p$ \\
      $\K$ & $\Q_p(\theta)$ \\
    $\O$ & $\Z_p[\theta]$\\
    $\p$ & the ideal in $\O$ generated by $\theta - 1$ \\
    $\U$ & the unit group of $\O$ \\
    $\U_j$ & $1 + \p^j$ for $j \ge 2$ \\
    $\sigma_i$ & the automorphism of $\K$ defined by $\theta \mapsto \theta^i$ \\
    $\G$ & the Galois group of $\K/\Q_p$ \\
    $\sigma$ & a fixed generator of $\G$ \\
    $\tau$ & $\sigma^k$ with $k$ fixed as above \\
    $\chi$ & $\U \to \U, u \mapsto u \tau(u)^{-1}$ \\
    $\Hom$ & $\Hom_{\theta}(\O \wedge \O, \O)$ \\
    $\Gamma_n$ & $(\p^{n-1})^\ell B^{-1}$ \\
    $\Delta_n$ & $\Gamma_n \setminus \Gamma_{n+1}$ \\
    $C_{n,e}(x)$ & the skeleton group defined by $x \in \K^\ell$ \\
    $\Sigma_{n,i}$ & $ \{ x \in \Gamma_n : (1, \tau)(x) = \omega^i x \}$ \\
    $\Omega_{n,i}$ & $\Sigma_{n,i} \setminus \Sigma_{n+1,i}$ 
\end{longtable}
\end{minipage}}

%%%%%%%%%%%%%%%%%%%%%%%%%%%%%%%%%%%%%%%%%%%%%%%%%%%%%%%%%%%%%%%%%%%%%%%%%%%%%
\section*{Funding}
\noindent This work was supported by the German Research Foundation  [DFG project 386837064 to A.\ C.\ and T.~M.]; and the Australian Research Council [DP190100317 to H.\ D.].

\enlargethispage{1ex}

\bibliographystyle{abbrv}

\def\cprime{$'$} \def\cprime{$'$}

\end{document}